\documentclass[runningheads,a4paper]{llncs}

\usepackage{amssymb}
\usepackage{amsmath}
\usepackage{graphicx}

\usepackage{url}
\urldef{\mailsc}\path|{ep374, cbs31}@cam.ac.uk|    
\newcommand{\keywords}[1]{\par\addvspace\baselineskip
\noindent\keywordname\enspace\ignorespaces#1}
\newcommand{\om}{\mathrm{\Omega}}

\usepackage{enumerate}
\usepackage{subcaption}

\usepackage{color}
\usepackage[usenames,dvipsnames,svgnames,table]{xcolor}
\usepackage{hyperref}
\hypersetup{
  colorlinks,
  citecolor= Brown,
  linkcolor= link,
  urlcolor= link
}
\definecolor{link}{rgb}{0.18,0.25,0.78}

 \usepackage{lineno}

  \numberwithin{equation}{section}

\begin{document}

\setlength{\belowdisplayskip}{2.8pt} \setlength{\belowdisplayshortskip}{2.8pt}
\setlength{\abovedisplayskip}{2.8pt} \setlength{\abovedisplayshortskip}{2.8pt}
\setlength{\belowcaptionskip}{-22pt}

\mainmatter  

\title{Infimal Convolution Regularisation Functionals of $\mathrm{BV}$ and $\mathrm{L}^{p}$ Spaces. The Case $p=\infty$
}

\titlerunning{Infimal convolution regularisation of $\mathrm{BV}$ and $\mathrm{L}^{\infty}$}

\author{Martin Burger \inst{1}
\and Konstantinos Papafitsoros \inst{2}\and Evangelos Papoutsellis \inst{3}\and Carola-Bibiane Sch\"onlieb \inst{3}}
\authorrunning{Burger, Papafitsoros, Papoutsellis and Sch\"onlieb }

\institute{
Institute for Computational and Applied Mathematics, University of M\"unster, Germany
\and Institute for Mathematics, Humboldt University of Berlin, Germany
\and Department of Applied Mathematics and Theoretical Physics, University of Cambridge, United Kingdom\\
\email{martin.burger@wwu.de}, 
\email{papafitsoros@hu-berlin.de},
\mailsc
}
%\thanks{The authors acknowledge...}
%\institute{Springer-Verlag, Computer Science Editorial,\\
%Tiergartenstr. 17, 69121 Heidelberg, Germany\\
%\url{http://www.springer.com/lncs}}


%\toctitle{Lecture Notes in Computer Science}
%\tocauthor{Authors' Instructions}
\maketitle

\begin{abstract}
In this paper we analyse an infimal convolution type regularisation functional called $\mathrm{TVL}^{\infty}$, based on the total variation ($\mathrm{TV}$) and the $\mathrm{L}^{\infty}$ norm of the gradient. The  functional belongs to a more general family of $\mathrm{TVL}^{p}$ functionals ($1<p\le \infty$) introduced in \cite{tvlp}. There, the case $1<p<\infty$ is examined while here we focus on the $p=\infty$ case. We show via analytical and numerical results that the minimisation of the  $\mathrm{TVL}^{\infty}$ functional promotes piecewise affine structures in the reconstructed images similar to the state of the art total generalised variation ($\mathrm{TGV}$) but improving upon preservation of hat--like structures. We also propose a spatially adapted version of our model that produces results comparable to $\mathrm{TGV}$ and allows space for further improvement.

\end{abstract}

\section{Introduction}
In the variational setting for imaging, given image data $f\in\mathrm{L}^{s}(\om)$, $\om\subseteq \mathbb{R}^{2}$, ones aims to reconstruct an image $u$ by minimising a functional of the type
\begin{equation}\label{min:general_T}
\min_{u\in X} \frac{1}{s}\|f-Tu\|_{\mathrm{L}^{s}(\om)}^{s}+\mathrm{\Psi}(u),
\end{equation}
over a suitable function space $X$. Here $T$  denotes a linear and bounded operator that encodes the transformation or degradation that the original image has gone through. Random noise is also usually present in the degraded image, the statistics of which determine the  norm  in the first term of \eqref{min:general_T}, the \emph{fidelity term}.  The presence of $\mathrm{\Psi}$, the \emph{regulariser},  makes the minimisation \eqref{min:general_T} a well--posed problem and its choice is crucial for the overall quality of the reconstruction. A classical regulariser in imaging is the total variation functional weighted with a parameter $\alpha>0$, $\alpha\mathrm{TV}$ \cite{rudin1992nonlinear}, where
\begin{equation}\label{def:tv}
\mathrm{TV}(u):=\sup\left \{\int_{\om}u\,\mathrm{div}\phi\,dx:\; \phi\in C_{c}^{1}(\om,\mathbb{R}^{2}), \; \|\phi\|_{\infty}\le 1 \right \}.
\end{equation}
While $\mathrm{TV}$ is able to preserve edges in the reconstructed image, it also promotes piecewise constant structures leading to undesirable staircasing artefacts. Several regularisers that incorporate higher order derivatives have been introduced in order to resolve this issue. The most prominent one has been the second order total generalised variation ($\mathrm{TGV}$) \cite{TGV} which can be interpreted  as a special type of infimal convolution of first and second order derivatives. Its definition reads
\begin{equation}\label{def:tgv}
\mathrm{TGV}_{\alpha,\beta}^{2}(u):=\min_{w\in\mathrm{BD}(\om)} \alpha \|Du-w\|_{\mathcal{M}}+\beta \|\mathcal{E}w\|_{\mathcal{M}}.
\end{equation}
Here $\alpha,\beta$ are positive parameters, $\mathrm{BD}(\om)$ is the space of functions of bounded deformation, $\mathcal{E}w$ is the distributional symmetrised gradient and $\|\cdot\|_{\mathcal{M}}$ denotes the Radon norm of a finite Radon measure, i.e.,
\[\|\mu\|_{\mathcal{M}}=\sup\left \{\int_{\om}\phi\,d\mu:\; \phi\in C_{c}^{\infty}(\om,\mathbb{R}^{2}), \; \|\phi\|_{\infty}\le 1 \right \}.\]
 $\mathrm{TGV}$ regularisation typically produces piecewise smooth reconstructions eliminating the staircasing effect. A plausible question  is whether results of similar quality can be achieved using simpler, first order regularisers. For instance, it is known that Huber $\mathrm{TV}$ can reduce the staircasing up to an extent \cite{Huber}.

In \cite{tvlp}, a family of first order infimal convolution type regularisation functionals is introduced, that reads
\begin{equation}\label{def:tvlp}
\mathrm{TVL}_{\alpha,\beta}^{p}(u):=\min_{w\in \mathrm{L}^{p}(\om)} \alpha \|Du-w\|_{\mathcal{M}}+\beta \|w\|_{\mathrm{L}^{p}(\om)},
\end{equation}
where $1<p\le \infty$. While in \cite{tvlp}, basic properties of \eqref{def:tvlp} are shown for the general case $1<p\le \infty$, see Proposition \ref{lbl:tvlinf_properties}, the main focus remains on the finite $p$ case. There, the $\mathrm{TVL}^{p}$ regulariser is successfully applied to image denoising and decomposition, reducing significantly the staircasing effect and producing piecewise smooth results that are very similar to the solutions obtained by $\mathrm{TGV}$. Exact solutions of the $\mathrm{L}^{2}$ fidelity denoising problem are also computed there for simple one dimensional data.

\subsection*{Contribution of the Present Work}
The purpose of the present paper is to examine more thoroughly the case $p=\infty$, i.e.,
\begin{equation}\label{tvlinf}
\mathrm{TVL}_{\alpha,\beta}^{\infty}(u):=\min_{w\in \mathrm{L}^{\infty}(\om)} \alpha \|Du-w\|_{\mathcal{M}}+\beta \|w\|_{\mathrm{L}^{\infty}(\om)},
\end{equation}
and the use of the $\mathrm{TVL}^{\infty}$ functional in $\mathrm{L}^{2}$ fidelity denoising 
\begin{equation}\label{min:general}
\min_{u} \frac{1}{2}\|f-u\|_{\mathrm{L}^{2}(\om)}^{2}+\mathrm{TVL}_{\alpha,\beta}^{\infty}(u).
\end{equation}
%We mainly study the problem \eqref{min:general} when $s=2$ and $T=Id$, i.e., denoising images corrupted by Gaussian noise. 
We  study thoroughly the one dimensional  version of  \eqref{min:general}, by computing exact solutions for data $f$ a piecewise constant and a piecewise affine step function. We show that the solutions are piecewise affine and we depict some similarities and  differences to TGV solutions. The functional $\mathrm{TVL}_{\alpha,\beta}^{\infty}$ is further tested for Gaussian denoising.  We show that  $\mathrm{TVL}_{\alpha,\beta}^{\infty}$, unlike $\mathrm{TGV}$, is able to recover hat--like structures, a property that is already present in the $\mathrm{TVL}_{\alpha,\beta}^{p}$ regulariser for large values of $p$, see \cite{tvlp}, and it is enhanced here.
 After explaining some limitations of our model, we propose an extension where the parameter $\beta$ is spatially varying, i.e., $\beta=\beta(x)$, and discuss a rule for selecting its values. The resulting denoised images are comparable to the TGV reconstructions and indeed the model has the potential to produce much better results.

\section{Properties of the $\mathrm{TVL}_{\alpha,\beta}^{\infty}$ Functional}
The following properties of the $\mathrm{TVL}_{\alpha,\beta}^{\infty}$ functional are shown in \cite{tvlp}. We refer the reader to  \cite{tvlp,papoutsellisphd}  for the corresponding proofs and to \cite{AmbrosioBV} for an introduction to the space of functions of bounded variation $\mathrm{BV}(\om)$.

\newtheorem{tvlinf_properties}{Proposition}
\begin{tvlinf_properties}[\cite{tvlp}]\label{lbl:tvlinf_properties}
Let $\alpha,\beta>0$, $d\ge 1$, let $\om\subseteq \mathbb{R}^{d}$ be an open, connected domain with Lipschitz boundary and define for $u\in \mathrm{L}^{1}(\om)$
\begin{equation}\label{tvinf_prop}
\mathrm{TVL}_{\alpha,\beta}^{\infty}(u):=\min_{w\in \mathrm{L}^{\infty}(\om)} \alpha \|Du-w\|_{\mathcal{M}}+\beta \|w\|_{\mathrm{L}^{\infty}(\om)}. 
\end{equation}
Then we have the following:
\begin{enumerate}[(i)]
\item $\mathrm{TVL}_{\alpha,\beta}^{\infty}(u)<\infty$ if and only if $u\in \mathrm{BV}(\om)$.
\item If $u\in \mathrm{BV}(\om)$ then the minimum in \eqref{tvinf_prop} is attained. 
\item $\mathrm{TVL}_{\alpha,\beta}^{\infty}(u)$ can equivalently be defined as
\[\mathrm{TVL}_{\alpha,\beta}^{\infty}(u)=\sup\left \{\int_{\om}u\,\mathrm{div}\phi\,dx:\; \phi\in C_{c}^{1}(\om,\mathbb{R}^{2}), \; \|\phi\|_{\infty}\le \alpha,\; \|\phi\|_{\mathrm{L}^{1}(\om)}\le \beta \right \},\]
and $\mathrm{TVL}_{\alpha,\beta}^{\infty}$ is lower semicontinuous w.r.t. the  strong $\mathrm{L}^{1}$ convergence.
\item There exist constants $0<C_{1}<C_{2}<\infty$ such that 
\[C_{2}\mathrm{TV}(u)\le \mathrm{TVL}_{\alpha,\beta}^{\infty}(u)\le C_{1}\mathrm{TV}(u),\quad \text{for all }u\in\mathrm{BV}(\om).\]
\item If $f\in \mathrm{L}^{2}(\om)$, then the minimisation problem
\[\min_{u\in\mathrm{BV}(\om)} \frac{1}{s}\|f-u\|_{\mathrm{L}^{2}(\om)}^{2}+\mathrm{TVL}_{\alpha,\beta}^{\infty}(u),\]
has a unique solution.
\end{enumerate}
\end{tvlinf_properties}

\section{The One Dimensional $\mathrm{L^{2}}$--$\mathrm{TVL}_{\alpha,\beta}^{\infty}$ Denoising Problem}

In order to get an intuition about the underlying regularising mechanism of the $\mathrm{TVL}_{\alpha,\beta}^{\infty}$ regulariser, we study here the one dimensional $\mathrm{L}^{2}$ denoising problem
\begin{equation}\label{min_1d}
\min_{\substack{u\in\mathrm{BV}(\om)  \\ w\in \mathrm{L}^{\infty}(\om) }} \frac{1}{2}\|f-u\|_{\mathrm{L}^{2}(\om)}^{2}+\alpha \|Du-w\|_{\mathcal{M}}+\beta \|w\|_{\mathrm{L}^{\infty}(\om)}.
\end{equation}
In particular, we  present  exact solutions for simple data functions. In order to do so, we use the following theorem:

\newtheorem{optimality}[tvlinf_properties]{Theorem}
\begin{optimality}[Optimality conditions]\label{lbl:optimality}
Let $f\in \mathrm{L}^{2}(\om)$. A pair $(u,w)\in \mathrm{BV}(\om)\times \mathrm{L}^{\infty}(\om)$ is a solution of \eqref{min_1d} if and only if there exists a unique $\phi\in \mathrm{H}_{0}^{1}(\om)$ such that 
\begin{align}
\phi'&= u-f,\label{opt1}\\
\phi&\in \alpha \mathrm{Sgn}(Du-w),\label{opt2}\\
\phi&\in \begin{cases}
\left \{\psi\in \mathrm{L}^{1}(\om):\; \|\psi\|_{\mathrm{L}^{1}(\om)}\le \beta \right \}, & \text{ if } w=0,\\
\left \{\psi\in \mathrm{L}^{1}(\om):\; \langle \psi, w \rangle=\beta\|w\|_{\mathrm{L}^{\infty}(\om)},\; \|\psi\|_{\mathrm{L}^{1}(\om)}\le \beta \right \}, & \text{ if }w\ne 0.
\end{cases}\label{opt3}
\end{align}
\end{optimality}
Recall that for a finite Radon measure $\mu$,  $\mathrm{Sgn}(\mu)$ is defined as 
\[\mathrm{Sgn}(\mu)=\left \{\phi\in \mathrm{L}^{\infty}(\om)\cap \mathrm{L}^{\infty}(\om,\mu):\; \|\phi\|_{\mathrm{L}^{\infty}(\om)}\le 1,\; \phi=\frac{d\mu}{d|\mu|}, \; |\mu|-\text{a.e.} \right \}.\]
As it is shown in \cite{BrediesL1}, $\mathrm{Sgn}(\mu)\cap C_{0}(\om)=\partial \|\cdot\|_{\mathcal{M}}(\mu)\cap C_{0}(\om)$.
\begin{proof}
 The proof of Theorem \ref{lbl:optimality} is based on Fenchel--Rockafellar duality theory and follows closely the corresponding proof of the finite $p$ case. We thus omit it and we refer the reader  to \cite{tvlp,papoutsellisphd} for further details, see also
\cite{BrediesL1,Papafitsoros_Bredies} for the analogue optimality conditions for the one dimensional $\mathrm{L}^{1}$--$\mathrm{TGV}$ and $\mathrm{L}^{2}$--$\mathrm{TGV}$ problems. 
\qed
\end{proof}

The following proposition states that the solution $u$ of \eqref{min_1d} is essentially piecewise affine.

\newtheorem{affine}[tvlinf_properties]{Proposition}
\begin{affine}[Affine structures]\label{lbl:affine}
Let $(u,w)$ be an optimal solution pair for \eqref{min_1d} and $\phi$ be the corresponding dual function. Then $|w|=\|w\|_{\mathrm{L}^{\infty}(\om)}$ a.e. in the set $\{\phi\ne 0\}$. Moreover, $|u'|=\|w\|_{\mathrm{L}^{\infty}(\om)}$ whenever $u>f$ or $u<f$.
\end{affine}
\begin{proof}
Suppose that there exists a $U\subseteq \{\phi\ne 0\}$ of positive measure such that $|w(x)|<\|w\|_{\mathrm{L}^{\infty}(\om)}$ for every $x\in U$. Then 
\begin{align*}
\int_{\om}\phi w\,dx &\le \int_{\om\setminus U} |\phi||w|\,dx +\int_{U}|\phi||w|\,dx<\|w\|_{\mathrm{L}^{\infty}(\om)}\left (\int_{\om\setminus U}|\phi|\,dx + \int_{U}|\phi|\,dx\right )\\
				&=\|w\|_{\mathrm{L}^{\infty}(\om)}\|\phi\|_{\mathrm{L}^{1}(\om)}=\beta\|w\|_{\mathrm{L}^{\infty}(\om)},
\end{align*}
where we used the fact that $\|\phi\|_{\mathrm{L}^{1}(\om)}\le \beta$ from \eqref{opt3}. However this contradicts the fact that $\langle\phi, w \rangle=\beta\|w\|_{\mathrm{L}^{\infty}(\om)}$ also from \eqref{opt3}. Note also that from \eqref{opt1} we have that $\{u>f\}\cup \{u<f\}\subseteq \{\phi\ne 0\}$ up to null sets. Thus, the last statement of the proposition follows from the fact that whenever $u>f$ or $u<f$ then $u'=w$ there. This last fact can be shown exactly as in the corresponding $\mathrm{TGV}$ problems, see \cite[Prop. 4.2]{BrediesL1}.
\qed
\end{proof}

Piecewise affinity is typically a characteristic of higher order regularisation models, e.g. $\mathrm{TGV}$. Indeed, as the next proposition shows, $\mathrm{TGV}$ and $\mathrm{TVL}^{\infty}$ regularisation coincide in some simple special cases.

\newtheorem{equi}[tvlinf_properties]{Proposition}
\begin{equi}\label{lbl:equi}
The one dimensional functionals $\mathrm{TGV}_{\alpha,\beta}^{2}$ and $\mathrm{TVL}_{\alpha,2\beta}^{\infty}$ coincide in the class of those $\mathrm{BV}$ functions $u$, for which an optimal $w$ in both definitions of $\mathrm{TGV}$ and $\mathrm{TVL}^{\infty}$ is odd and monotone.
\end{equi}
\begin{proof}
Note first that for every odd and monotone bounded function $w$ we have $\|Dw\|_{\mathcal{M}}=2\|w\|_{\mathrm{L}^{\infty}(\om)}$ and denote this set of functions by $\mathcal{A}\subseteq \mathrm{BV}(\om)$. For a $\mathrm{BV}$ function $u$ as in the statement of the proposition we have
\begin{align*}
\mathrm{TGV}_{\alpha,\beta}^{2}(u)&=\underset{w\in \mathcal{A}}{\operatorname{argmin}}\; \alpha\|Du-w\|_{\mathcal{M}}+\beta\|Dw\|_{\mathcal{M}}\\
&=\underset{w\in \mathcal{A}}{\operatorname{argmin}}\; \alpha\|Du-w\|_{\mathcal{M}}+2\beta\|w\|_{\mathrm{L}^{\infty}(\om)}
= \mathrm{TVL}_{\alpha,2\beta}^{\infty}(u). 
\end{align*}
\qed
\end{proof}
\textbf{Exact solutions}: We present exact solutions for the minimisation problem \eqref{min_1d}, for two simple functions $f,g:(-L,L)\to \mathbb{R}$ as data, where $f(x)=h\mathcal{X}_{(0,L)}(x)$ and $g(x)=f(x)+\lambda x$, with $\lambda,h>0$. Here $\mathcal{X}_{C}(x)=1$ for $x\in C$ and $0$ otherwise.

\begin{figure}[t!]
\begin{center}
\begin{subfigure}[t]{0.45\textwidth}
	\centering
	\includegraphics[height=0.94\textwidth]{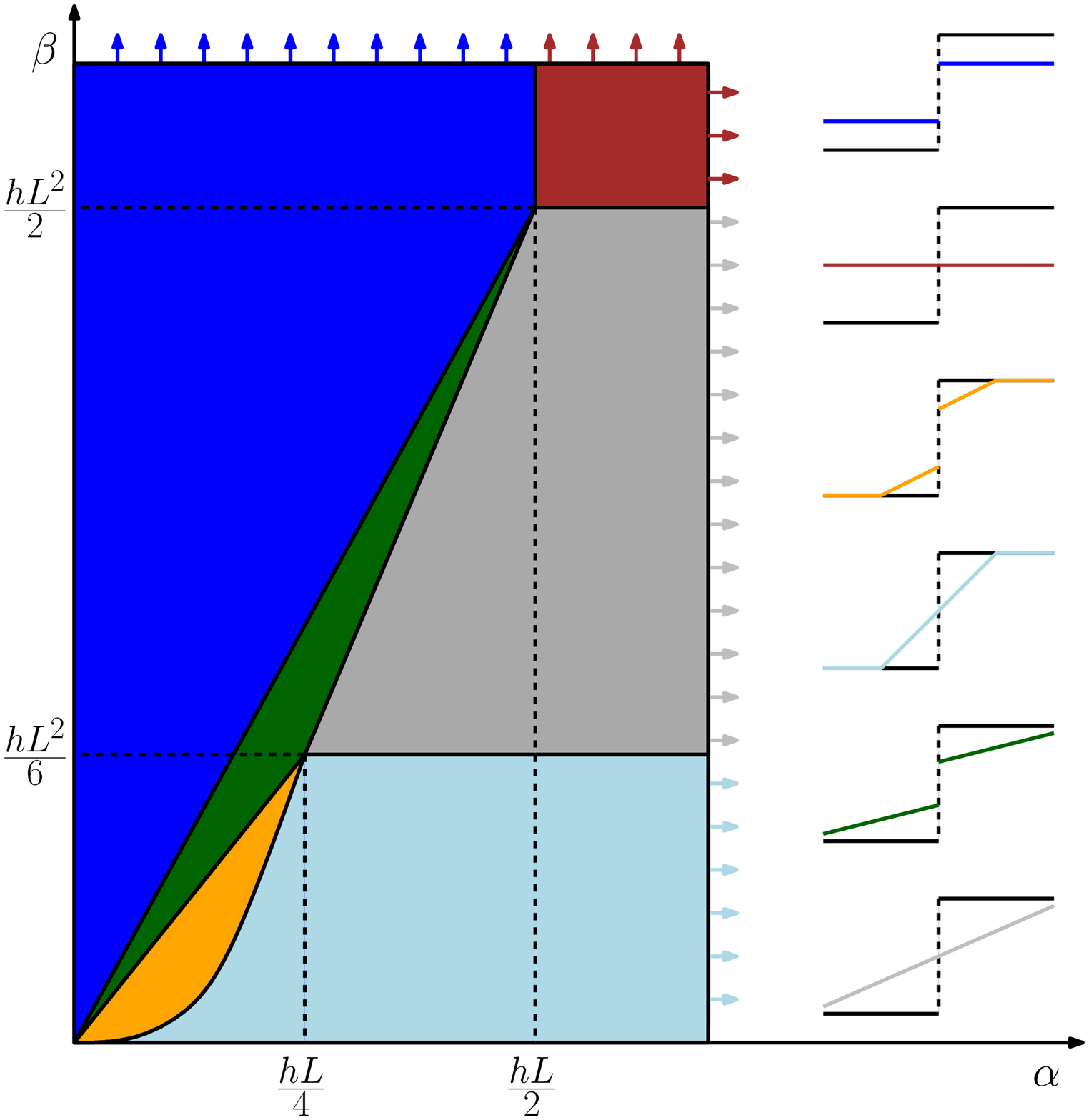}
	\caption{\centering Piecewise constant data function $f$}
	\label{fig:exact:a}
\end{subfigure}
\begin{subfigure}[t]{0.45\textwidth}
	\centering
	\includegraphics[height=0.94\textwidth]{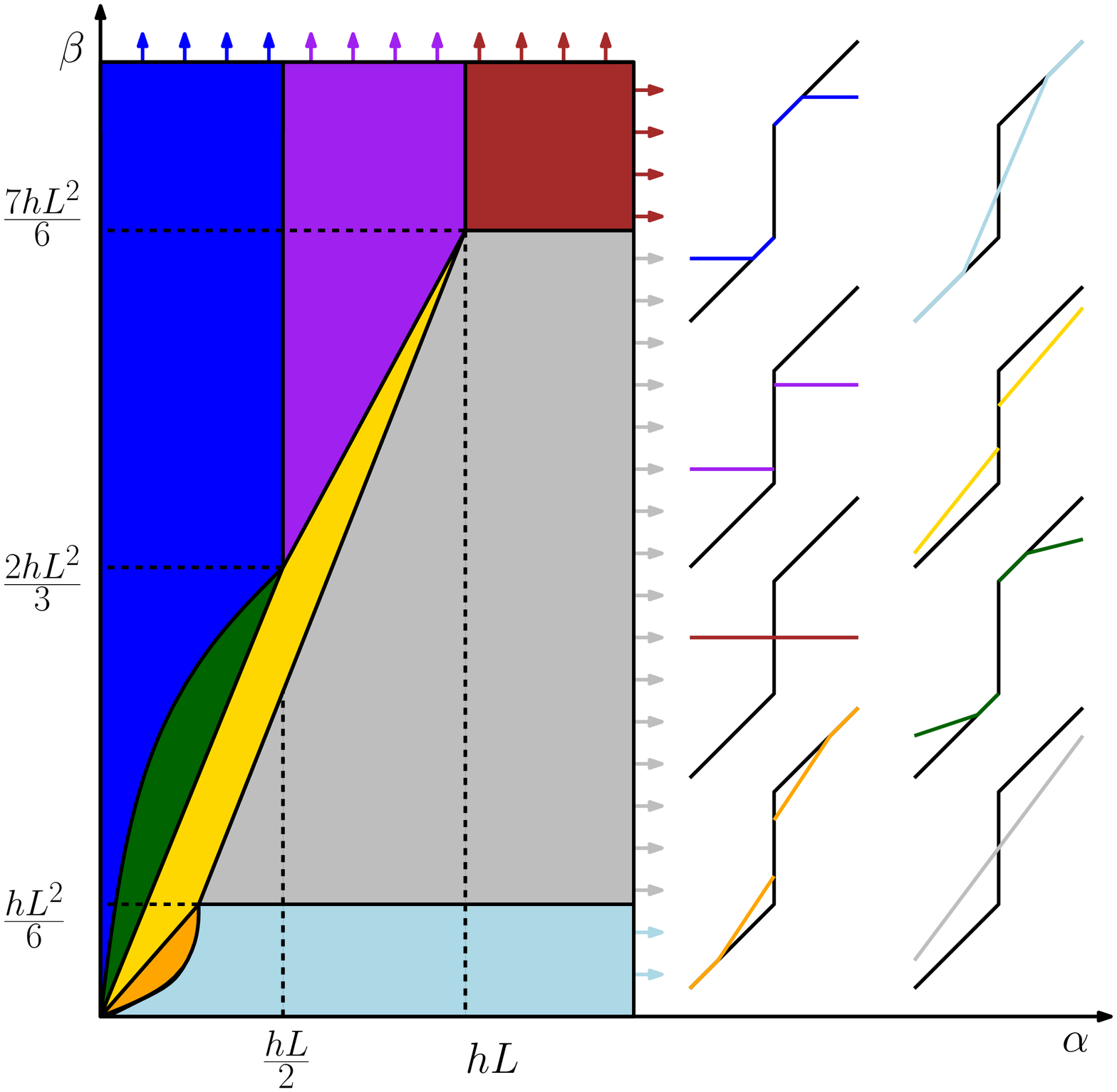}
	\caption{\centering Piecewise affine data function $g$}
	\label{fig:exact:b}
\end{subfigure}\\[0.8cm]
\caption{Exact solutions for the $\mathrm{L}^{2}$--$\mathrm{TVL}_{\alpha,\beta}^{\infty}$ one dimensional denoising problem}
\label{fig:exact}
\end{center}
\end{figure}

With the help of the optimality conditions \eqref{opt1}--\eqref{opt3} we are able to compute all possible solutions of \eqref{min_1d} for data $f$ and $g$ and for all values of $\alpha$ and $\beta$. These solutions are depicted in Figure \ref{fig:exact}. Every coloured region corresponds to a different type of solution. Note that there are regions where the solutions coincide with the corresponding solutions of  $\mathrm{TV}$ minimisation, see the blue and  red regions in Figure \ref{fig:exact:a} and the blue, purple  and red regions in Figure \ref{fig:exact:b}. This is not surprising since as it is shown in \cite{tvlp} for all dimensions, $\mathrm{TVL}_{\alpha,\beta}^{\infty}=\alpha\mathrm{TV}$, whenever $\beta/\alpha\ge |\om|^{1/q}$ with $1/p+1/q=1$ and $p\in (1,\infty]$. Notice also the presence of affine structures in all solutions, predicted by Proposition \ref{lbl:affine}. For demonstration purposes, we present the computation of the exact solution that corresponds to the yellow region of Figure \ref{fig:exact:b} and refer to \cite{papoutsellisphd} for the rest. 

Since we require a piecewise affine solution, from symmetry and \eqref{opt1} we have that $\phi(x)=(c_{1}-\lambda)\frac{x^{2}}{2}-c_{2}|x|+c_{3}$. Since we require $u$ to have a discontinuity at $0$, \eqref{opt2} implies $\phi(0)=a$ while from the fact that $\phi\in \mathrm{H}_{0}^{1}(\om)$ and from \eqref{opt3} we must have $\phi(-L)=0$ and $\langle \phi,w \rangle=\beta \|w\|_{\mathrm{L}^{\infty}(\om)}$. These conditions give
\[c_{1}=\frac{6(\alpha L-\beta)}{L^{3}}+\lambda,\quad c_{2}=\frac{4\alpha L-3\beta}{L^{2}},\quad c_{3}=\alpha.\]
We also have $c_{1}=u'=w$ and thus we require $c_{1}>0$. Since we have a jump at $x=0$ we also require $g(0)<u(0)<h$, i.e., $0<c_{2}<\frac{h}{2}$ and  $u(-L)>g(-L)$ i.e., $\phi'(-L)>0$. These last inequalities are translated to the following conditions
\[\left \{\beta<\alpha L+\frac{\lambda L^{3}}{6},\; \beta>\frac{4\alpha L}{3}-\frac{hL^{2}}{6},\; \beta>\frac{2\alpha L}{3},\; \beta<\frac{4\alpha L}{3} \right \},\]
which define the yellow area in  Figure \ref{fig:exact:b}. We can easily compute $u$ now:
\[
u(x)=
\begin{cases}
\big (\frac{6(\alpha L -\beta)}{L^{3}} +\lambda\big )x+h-\frac{4\alpha L-3\beta}{L^{2}},& \; x\in(0,L),\\
\big (\frac{6(\alpha L -\beta)}{L^{3}} +\lambda\big )x+\frac{4\alpha L-3\beta}{L^{2}},& \; x\in(-L,0).
\end{cases}
\]
Observe that when $\beta=\alpha L$, apart from the discontinuity, we can also recover the slope of the data $g'=\lambda$, something that neither  $\mathrm{TV}$ nor  $\mathrm{TGV}$ regularisation can give for this example, see \cite[Section 5.2]{Papafitsoros_Bredies}.

\section{Numerical Experiments}

In this section we present our numerical experiments for the discretised version of $\mathrm{L}^{2}$--$\mathrm{TVL}_{\alpha,\beta}^{\infty}$ denoising. We solve \eqref{min_1d} using the split Bregman algorithm, see \cite[Chapter 4]{papoutsellisphd} for more details. 

%The corresponding subproblems are solved exactly, i.e., a linear system solved by discrete cosine transform, a shrinkage operator and a projection onto the $\ell^{1}$ ball, . Also, for the spatially adapted $\mathrm{TVL}^{\infty}$ case discussed below, we use a weighted $\ell_{1}$ ok
%projection, see \cite{weighted_l1_Kopsinis}.

\begin{figure}[t!]
\begin{center}
\begin{subfigure}[t]{0.4\textwidth}
	\centering
	\includegraphics[width=0.9\textwidth]{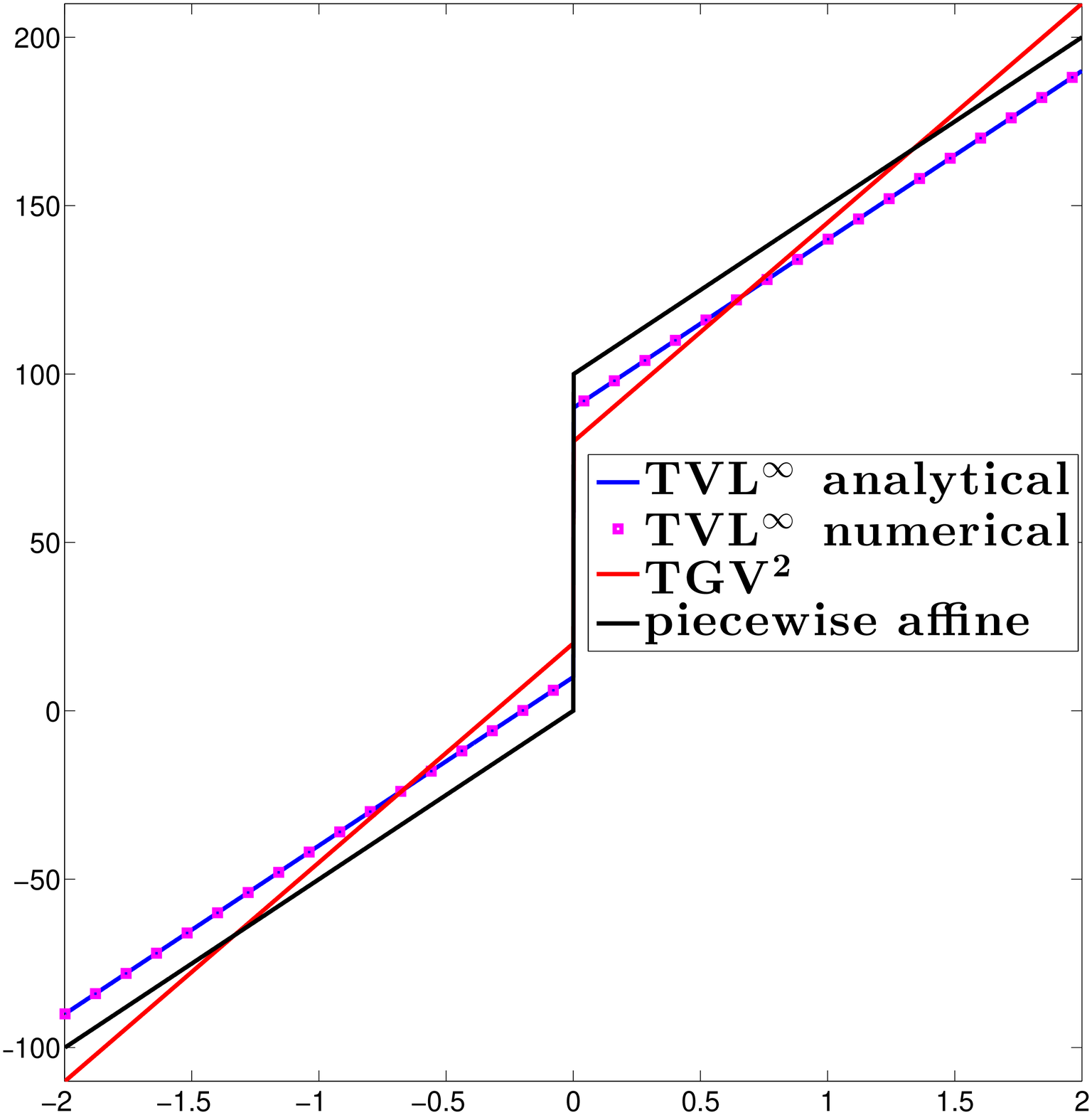}
	\caption{\centering Exact $\mathrm{TVL}^{\infty}$ and $\mathrm{TGV}$ solutions for the piecewise affine function $g$}
	\label{fig:1d_1}
\end{subfigure}
\begin{subfigure}[t]{0.4\textwidth}
	\centering
	\includegraphics[width=0.9\textwidth]{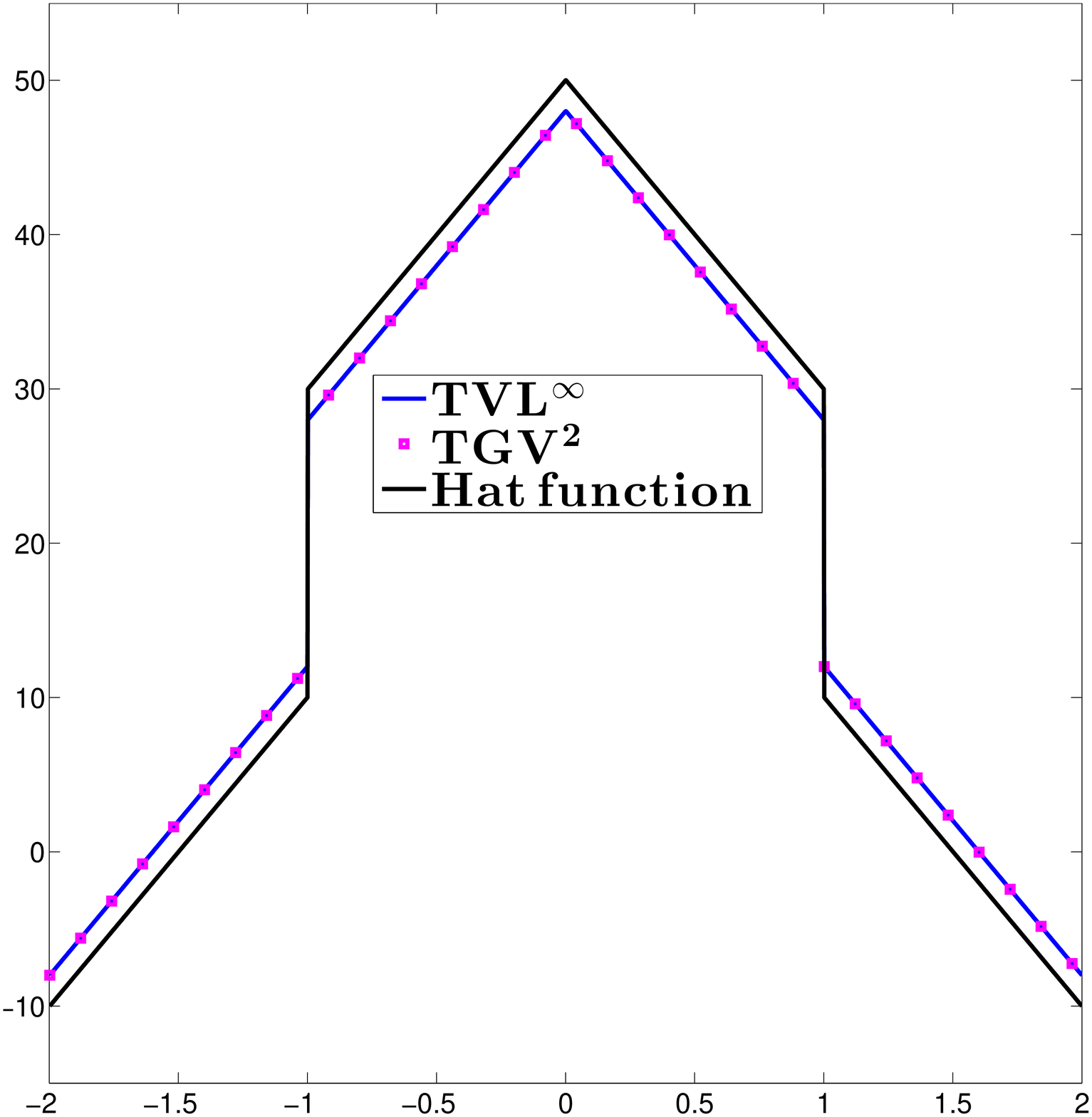}
	\caption{\centering Equivalence of $\mathrm{TVL}^{\infty}$ and $\mathrm{TGV}$ regularisation predicted by Proposition \ref{lbl:equi}}
	\label{fig:1d_2}
\end{subfigure}\\[0.8cm]
\caption{One dimensional numerical examples}
\label{fig:exactnum}
\end{center}
\end{figure}

First we present some one dimensional  examples that verify numerically our analytical results. Figure \ref{fig:1d_1}  shows the $\mathrm{TVL}^{\infty}$ result for the function $g(x)=h\mathcal{X}_{(0,L)}(x)+\lambda x$ where $\alpha, \beta$ belong to the yellow region of Figure \ref{fig:exact:b}. Note that the numerical and the analytical results coincide. We have also computed the $\mathrm{TGV}$ solution where the parameters are selected so that $\|f-u_{\mathrm{TGV}}\|_{2}=\|f-u_{\mathrm{TVL}^{\infty}}\|_{2}$.
 Figure \ref{fig:1d_2} shows a numerical verification of Proposition \ref{lbl:equi}. There, the $\mathrm{TVL}^{\infty}$ parameters are $\alpha=2$ and $\beta=4$, while the $\mathrm{TGV}$ ones are $\alpha=2$ and $\beta=2$. Both solutions coincide since they satisfy the symmetry properties of the proposition.

\begin{figure}[t!]
\begin{center}
\begin{subfigure}[t]{0.195\textwidth}
	\centering
	\includegraphics[width=0.9\textwidth]{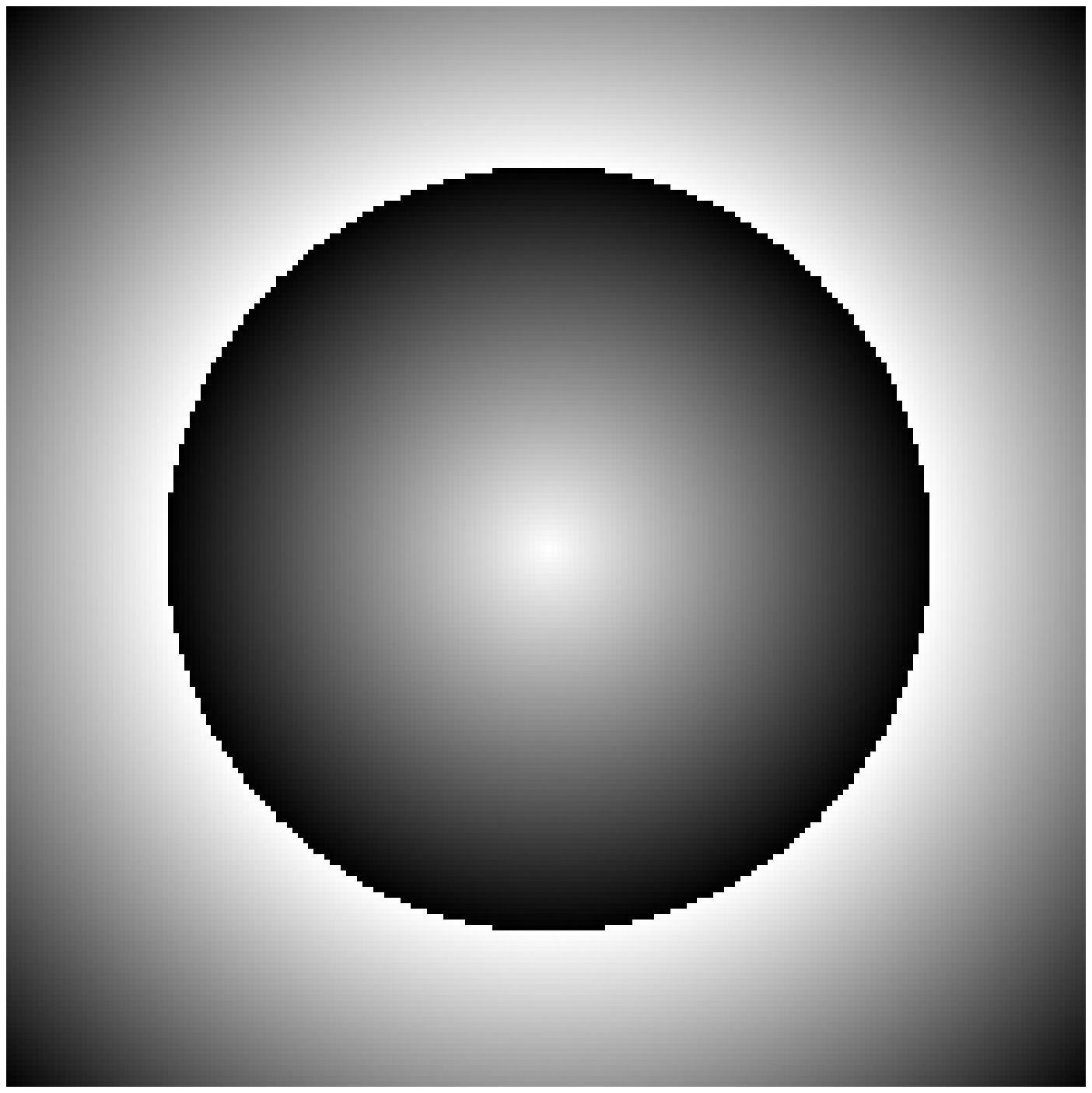}
	\label{fig:circle_clean}
	\caption{\centering Original synthetic image}
	\label{fig:circle:clean}
\end{subfigure}
\begin{subfigure}[t]{0.195\textwidth}
	\centering
	\includegraphics[width=0.9\textwidth]{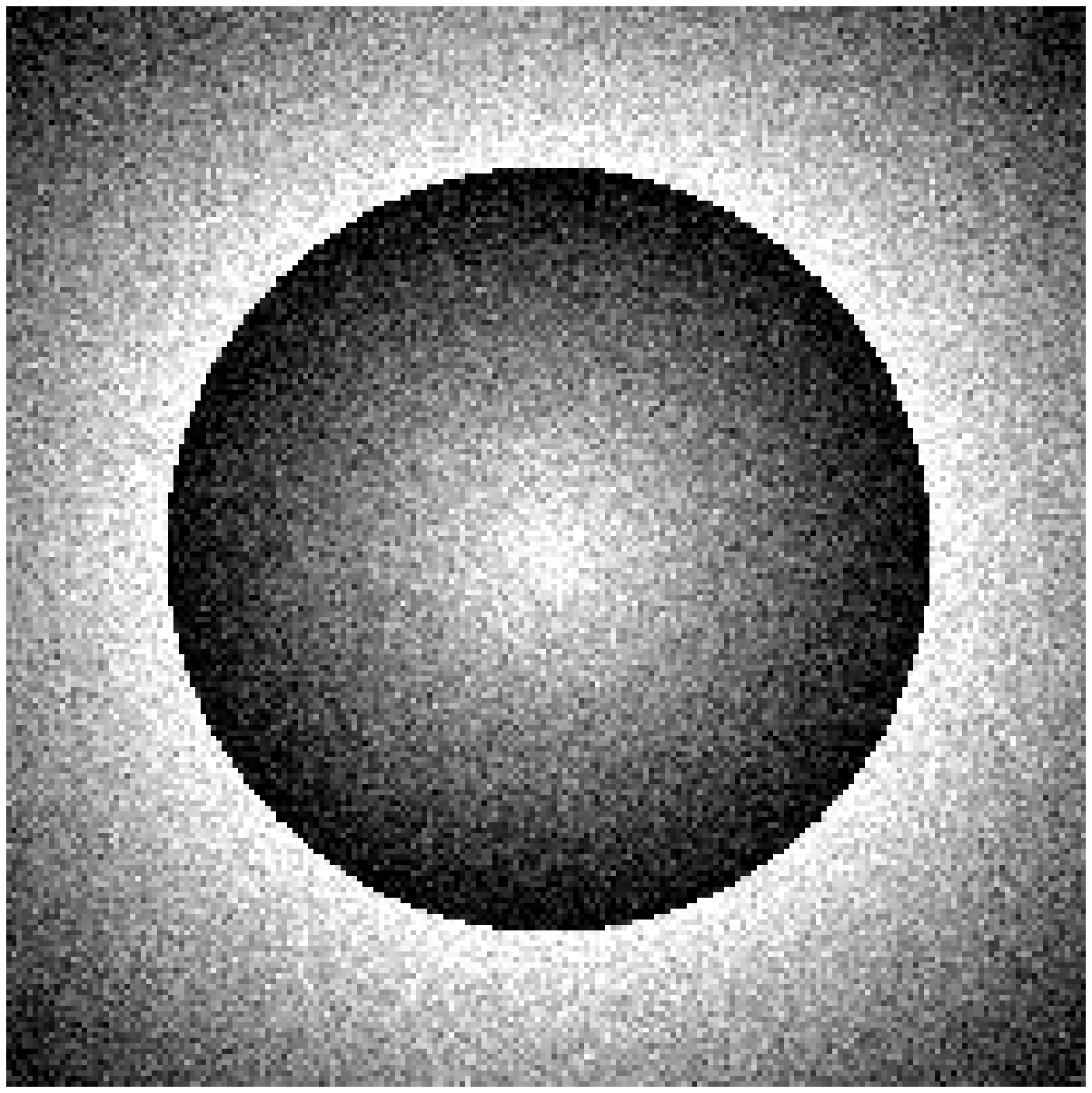}
	\label{fig:circle_noisy}
	\caption{\centering Noisy, Gaussian noise, $\sigma=0.01$ $\mathrm{SSIM}=0.2457$}
	\label{fig:circle:noisy}
\end{subfigure}
\begin{subfigure}[t]{0.195\textwidth}
	\centering
	\includegraphics[width=0.9\textwidth]{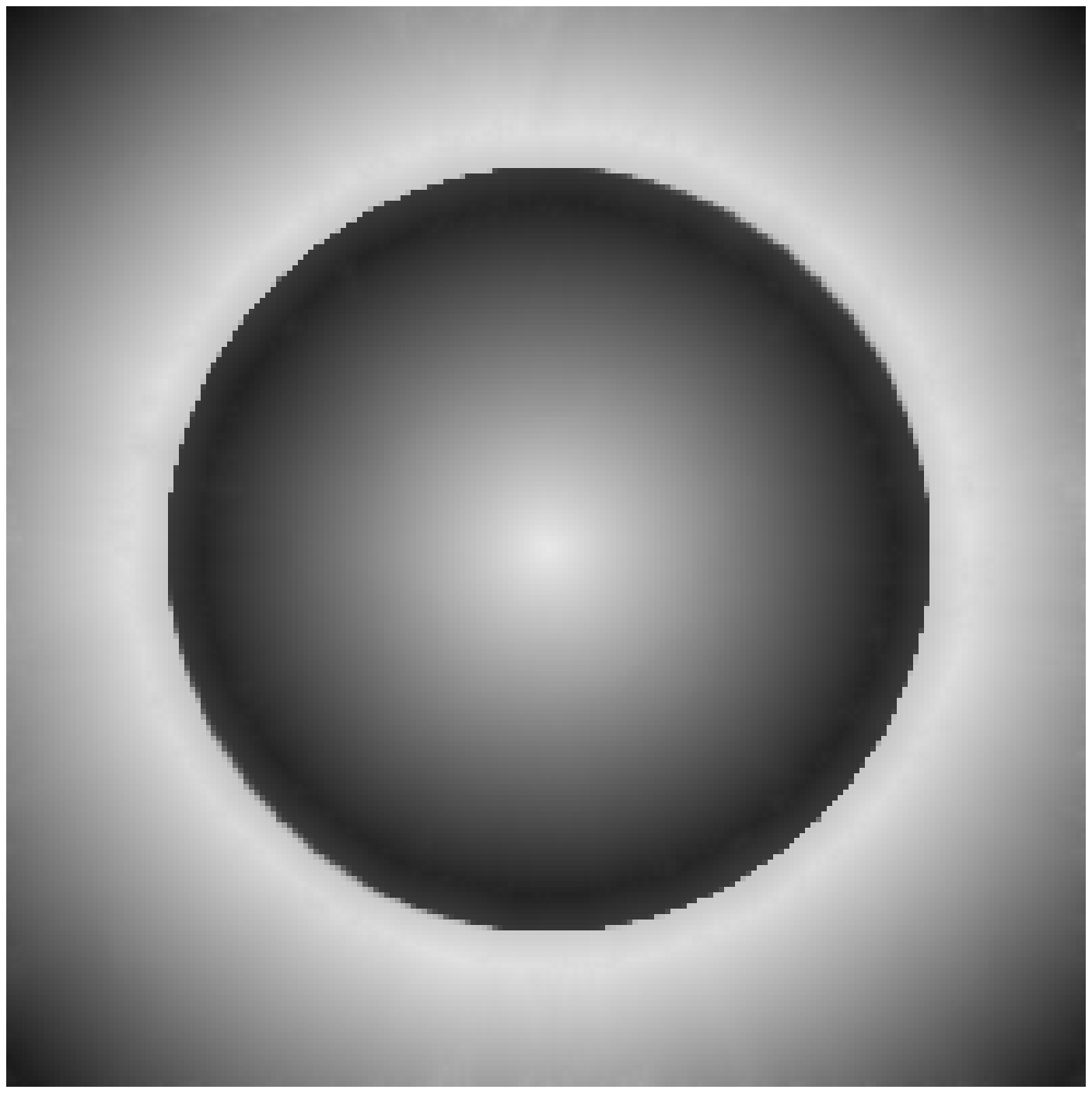}
	\label{fig:circle_tvlinf}
	\caption{\centering $\mathrm{TVL}^{\infty}$: $\alpha=0.7$, $\beta$=14000, $\mathrm{SSIM}			=0.9122$}
	\label{fig:circle:tvlinf}
\end{subfigure}
\begin{subfigure}[t]{0.19\textwidth}
	\centering
	\includegraphics[width=0.9\textwidth]{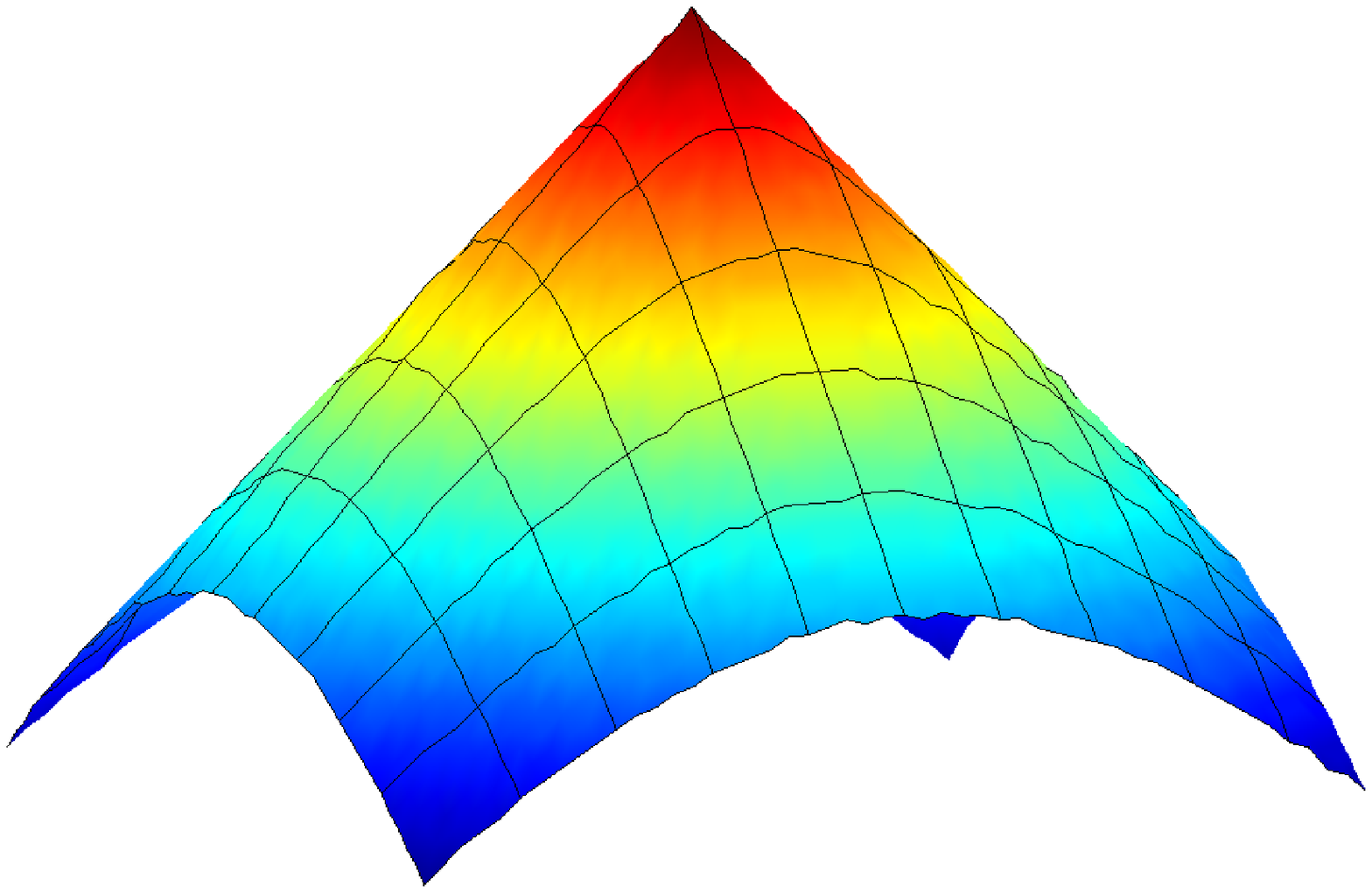}
	\label{fig:circle_tgvbreg_surf}
	\caption{\centering Original surface plot, central part zoom}
	\label{fig:circle:cleansurf}
\end{subfigure}
\begin{subfigure}[t]{0.19\textwidth}
	\centering
	\includegraphics[width=0.9\textwidth]{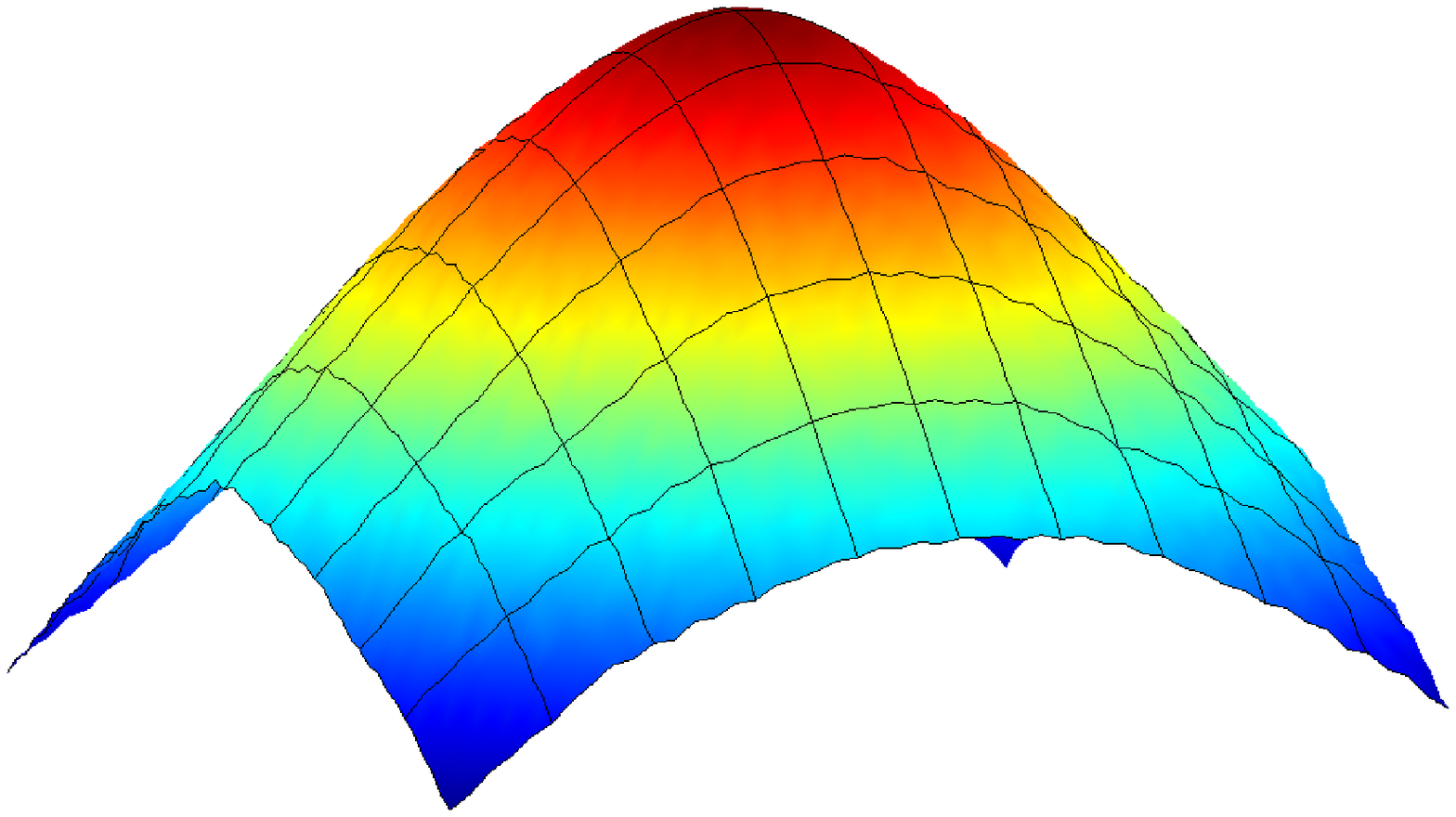}
	\label{fig:circle_tgvbreg_surf}
	\caption{\centering Bregman iteration $\mathrm{TGV}$ \newline  surface plot, central part zoom}
	\label{fig:circle:tgvsurf}
\end{subfigure}
\\[1cm]
\begin{subfigure}[t]{0.195\textwidth}
	\centering
	\includegraphics[width=0.9\textwidth]{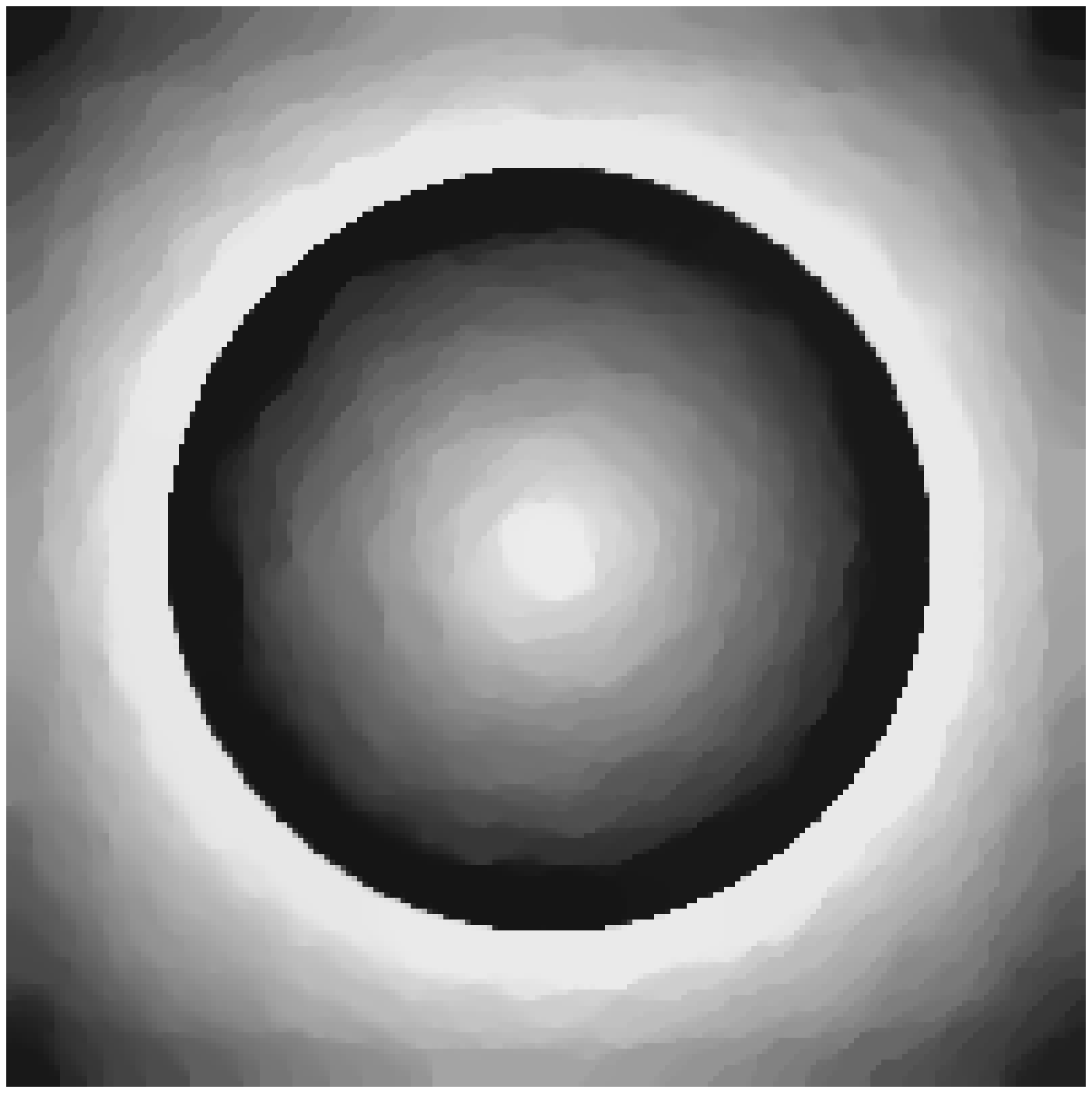}
	\label{fig:circle_tvbreg}
	\caption{\centering Bregman iteration $\mathrm{TV}$: $\alpha=2$, $\mathrm{SSIM}=0.8912$}
	\label{fig:circle:tv}
\end{subfigure}
\begin{subfigure}[t]{0.195\textwidth}
	\centering
	\includegraphics[width=0.9\textwidth]{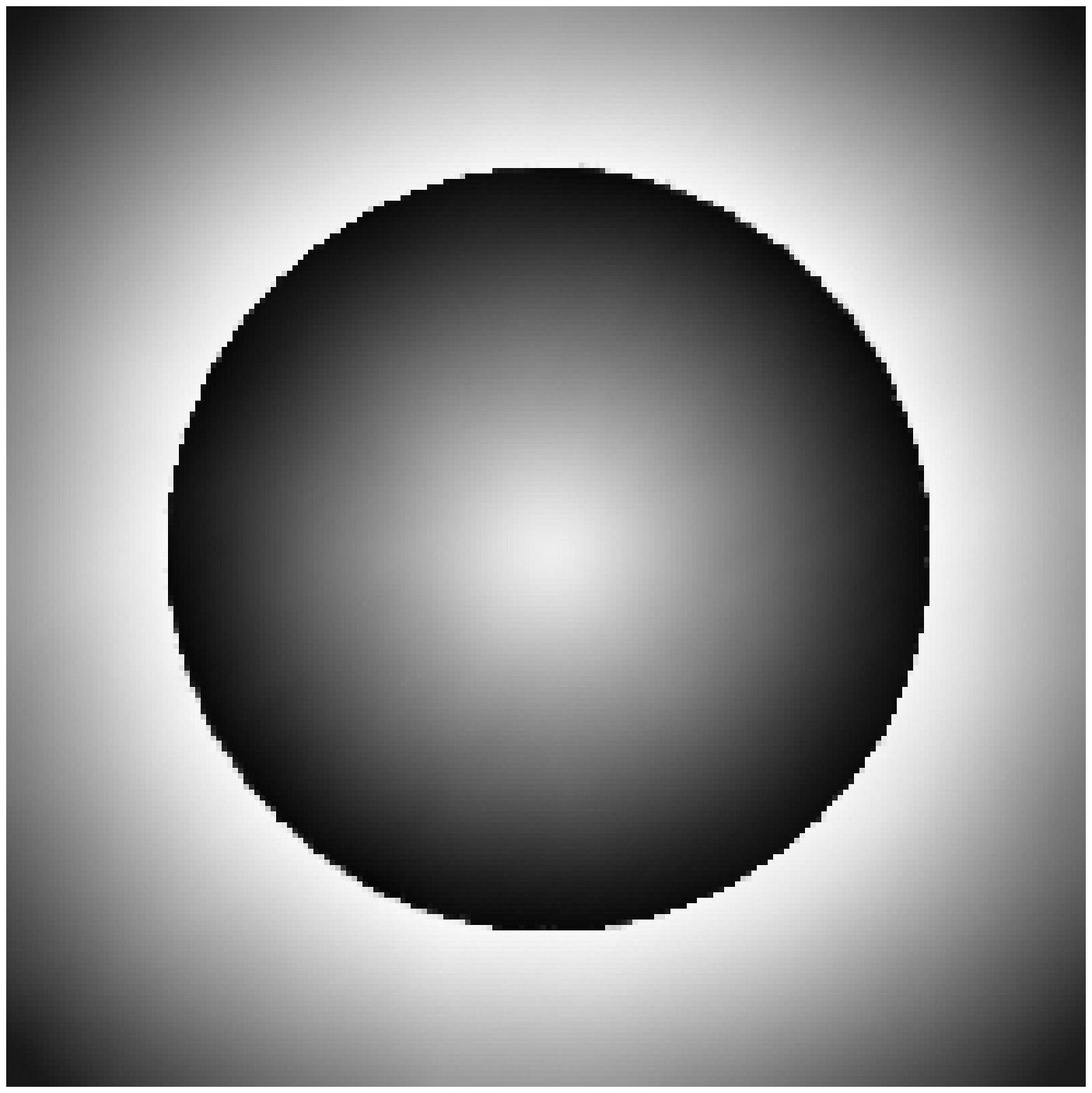}
	\label{fig:circle_tgvbreg}
	\caption{\centering Bregman iteration $\mathrm{TGV}$:\newline $\alpha=2$, $\beta=10$, $\mathrm{SSIM}			=0.9913$}
	\label{fig:circle:tgv}
\end{subfigure}
\begin{subfigure}[t]{0.195\textwidth}
	\centering
	\includegraphics[width=0.9\textwidth]{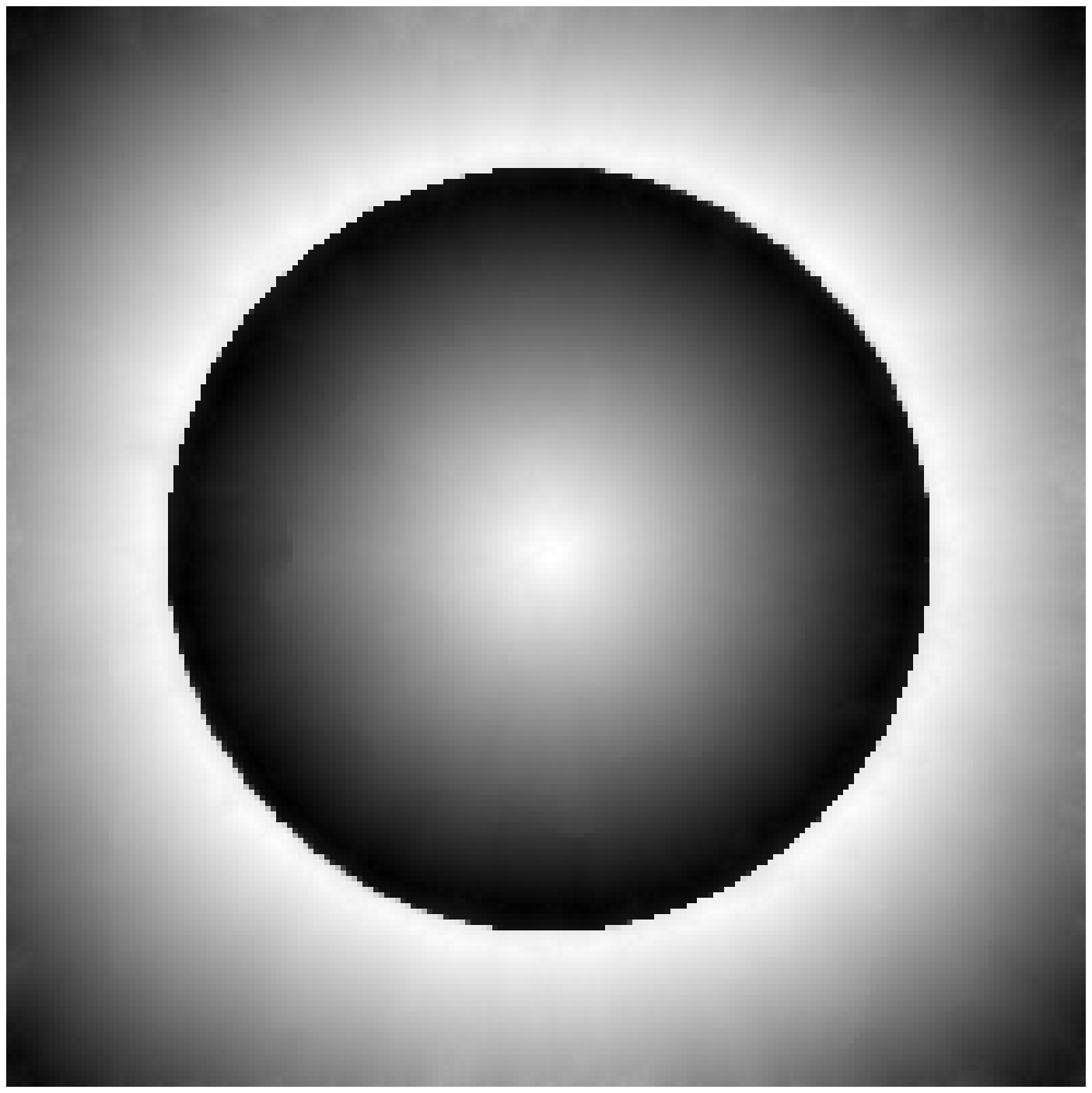}
	\label{fig:fig:circle_tvlinfbreg}
	\caption{\centering Bregman iteration $\mathrm{TVL}^{\infty}$: \newline $\alpha=3$, $\beta=65000$, \newline$				\mathrm{SSIM}=0.9828$}
	\label{fig:circle:tvlinfbreg}
\end{subfigure}
\begin{subfigure}[t]{0.19\textwidth}
	\centering
	\includegraphics[width=0.9\textwidth]{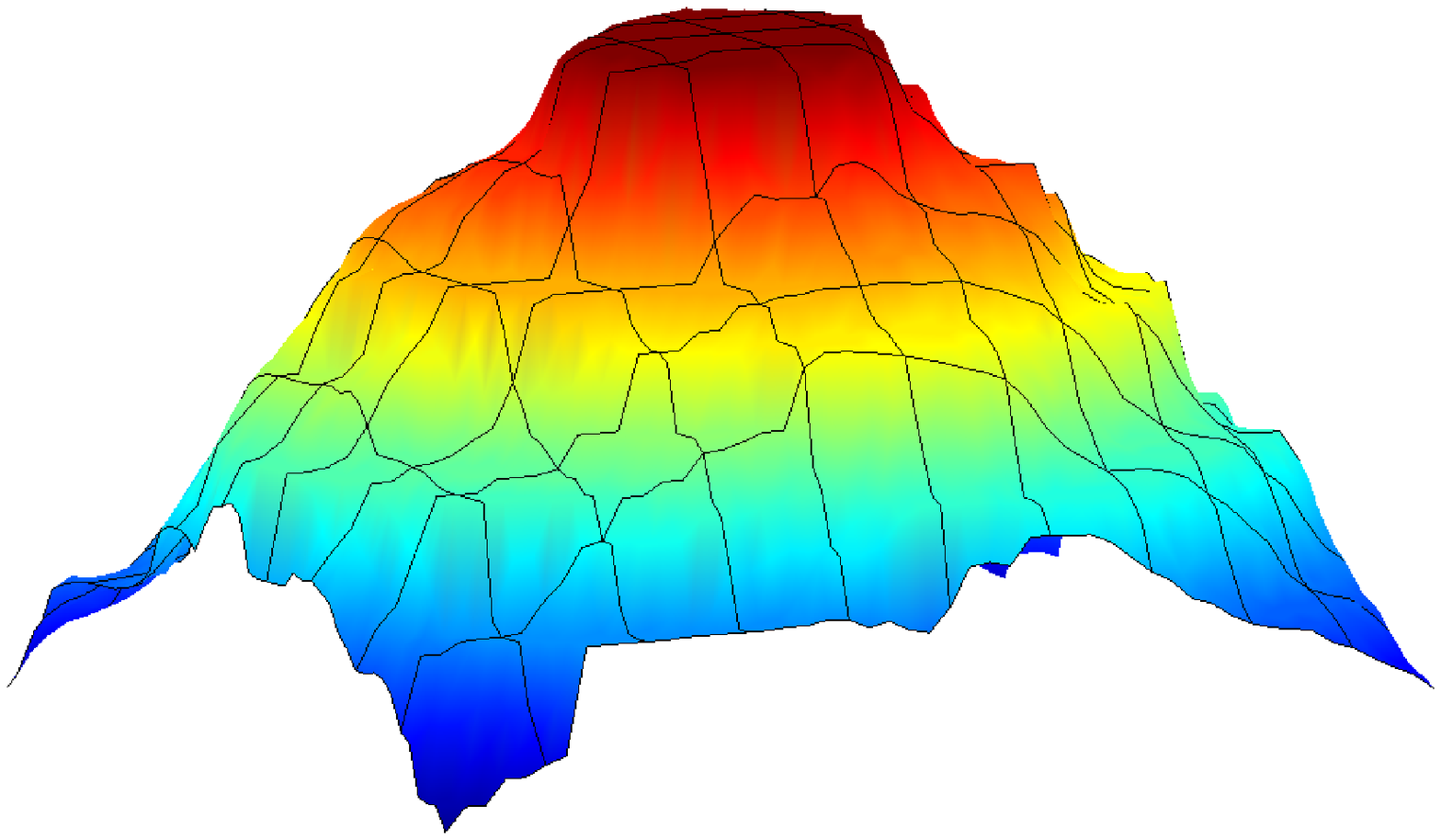}		
	\label{fig:fig:circle_tvlinfbreg_surf}
	\caption{\centering Bregman iteration $\mathrm{TV}$ surface plot, central part zoom}
	\label{fig:circle:tvsurf}
\end{subfigure}
\begin{subfigure}[t]{0.19\textwidth}
	\centering
	\includegraphics[width=0.9\textwidth]{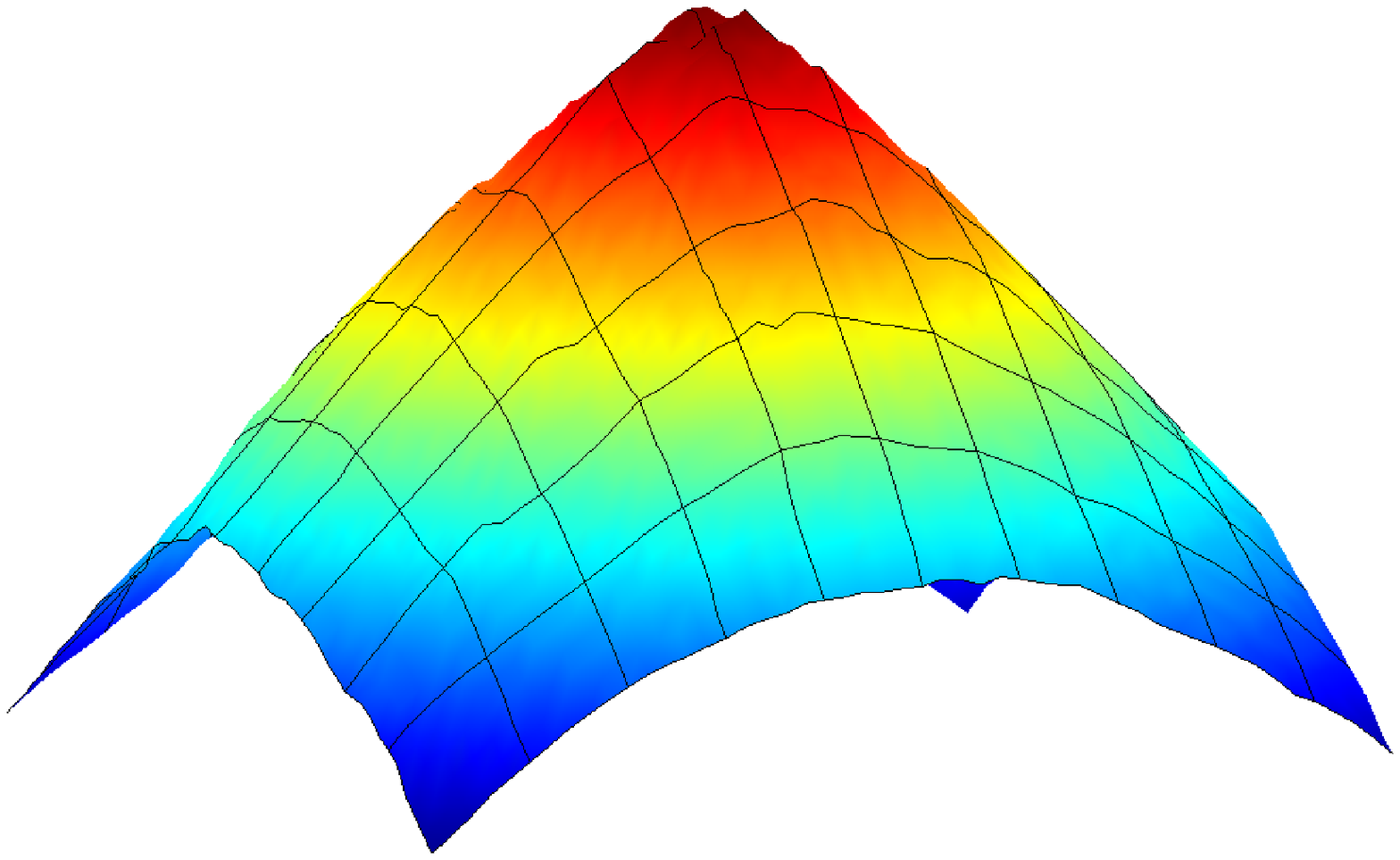}		
	\label{fig:fig:circle_tvlinfbreg_surf}
	\caption{\centering Bregman iteration $\mathrm{TVL}^{\infty}$ surface plot, central part zoom}
	\label{fig:circle:tvlinfsurf}
\end{subfigure}\\[0.8cm]
\caption{$\mathrm{TVL}^{\infty}$ based denoising and comparison with the corresponding $\mathrm{TV}$ and $\mathrm{TGV}$ results. All the parameters have been optimised for best SSIM.} %\cite{wang2004image}}
\label{fig:circle}
\end{center}
\end{figure}

We proceed now to two dimensional examples, starting from  Figure \ref{fig:circle}. There we used a synthetic image corrupted by additive Gaussian noise of $\sigma=0.01$, cf. Figures \ref{fig:circle:clean}--\ref{fig:circle:noisy}. We observe that  $\mathrm{TVL}^{\infty}$  denoises the image in a staircasing--free way in Figure \ref{fig:circle:tvlinf}. In order to do so however, one has to use large values of $\alpha$ and $\beta$ something that leads to a loss of contrast. This can easily be treated by solving the \emph{Bregman iteration} version of $\mathrm{L}^{2}$--$\mathrm{TVL}_{\alpha,\beta}^{\infty}$ minimisation, that is
\begin{equation}
\begin{aligned}
u^{k+1}&=\underset{u,w}{\operatorname{argmin}}\; \frac{1}{2} \|f-v^{k}-u\|_{2}^{2}+\alpha \|\nabla u -w\|_{1}+\beta \|w\|_{\infty},\\
v^{k+1}&=v^{k}+f-u^{k+1}.
\end{aligned}
\label{bregmanised}
\end{equation}
Bregman iteration has been widely used to deal with the loss of contrast in these type of regularisation methods, see \cite{TGVbregman,OBG} among others. For fair comparison we also employ the Bregman iteration version of $\mathrm{TV}$ and $\mathrm{TGV}$ regularisations. The Bregman iteration version of $\mathrm{TVL}^{\infty}$ regularisation produces visually a very similar result to the Bregman iteration version of $\mathrm{TGV}$, even though it has a slightly smaller SSIM value, cf. Figures \ref{fig:circle:tgv}--\ref{fig:circle:tvlinfbreg}. However,  $\mathrm{TVL}^{\infty}$  is able to reconstruct better the sharp spike at the middle of the figure, cf. Figures \ref{fig:circle:cleansurf}, \ref{fig:circle:tgvsurf} and \ref{fig:circle:tvlinfsurf}.

The reason for being able to obtain good reconstruction results with this particular example is due to the fact that the modulus of the gradient is essentially constant apart from the jump points. This is favoured by the $\mathrm{TVL}^{\infty}$ regularisation which promotes constant gradients as it is proved rigorously in dimension one in Proposition \ref{lbl:affine}. We expect that a similar analytic result holds in higher dimensions and we leave that for future work. However, this  is restrictive when gradients of different magnitude exist, see Figure \ref{fig:square}.
\begin{figure}[t!]
\begin{center}
\begin{subfigure}[t]{0.193\textwidth}
                \centering                                                  
                \includegraphics[width=0.9\textwidth]{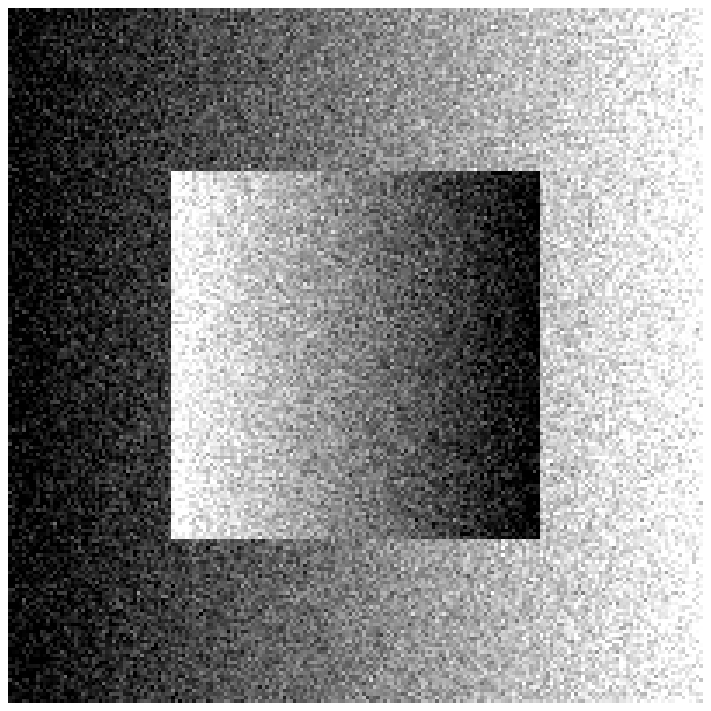}
                \caption{\centering Noisy, Gaussian noise, $\sigma=0.01$, $\mathrm{SSIM}=0.1791$}
                \label{square_infty:a}
\end{subfigure}
\begin{subfigure}[t]{0.193\textwidth}
                \centering                                                  
                \includegraphics[width=0.9\textwidth]{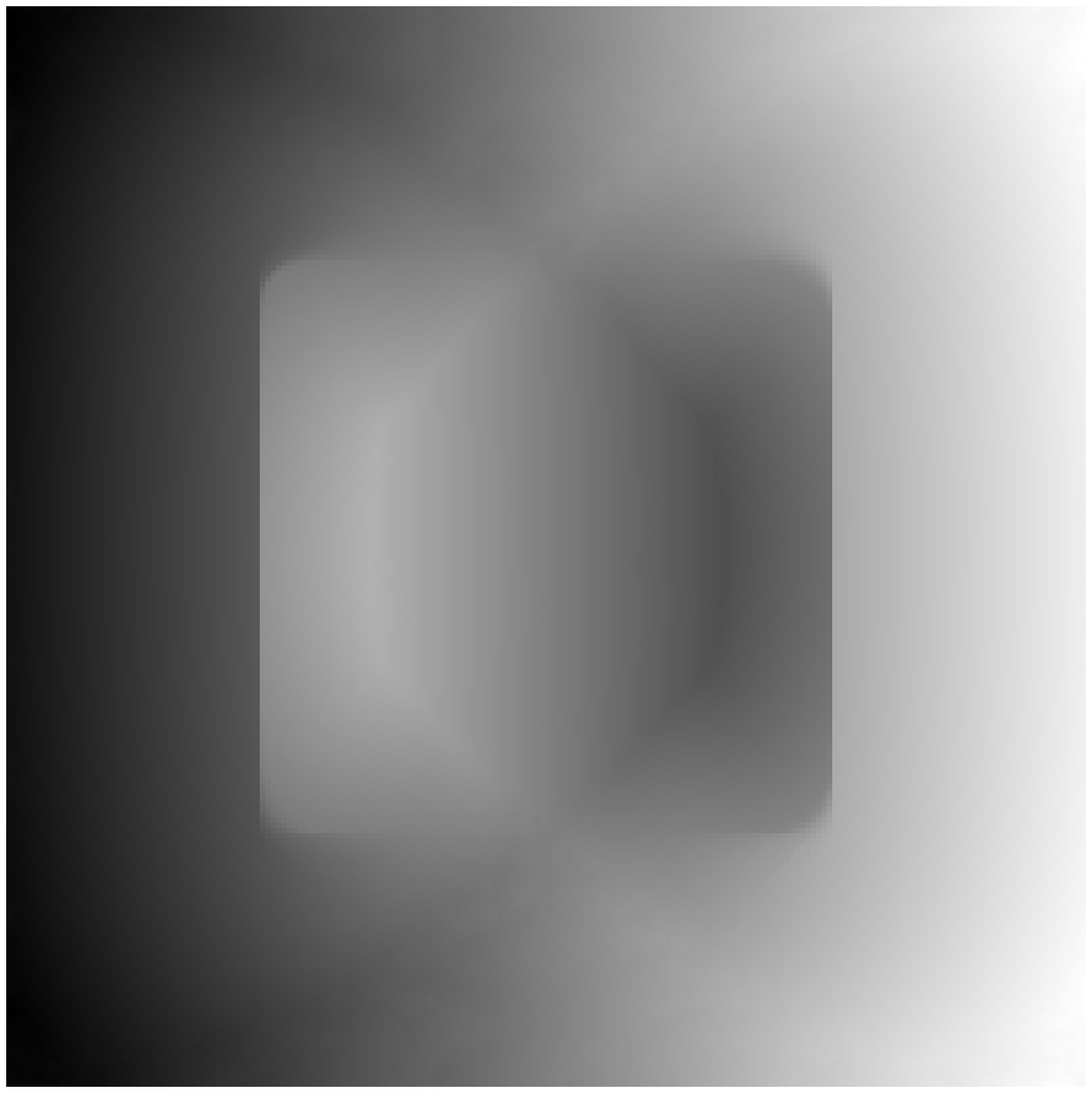}
                \caption{\centering $\mathrm{TVL}^{\infty}$: $\alpha=5$, $\beta=60000$, $\mathrm{SSIM}				=0.8197$}
                \label{square_infty:a}
\end{subfigure}
\begin{subfigure}[t]{0.193\textwidth}
                \centering                                                  
                \includegraphics[height=0.9\textwidth]{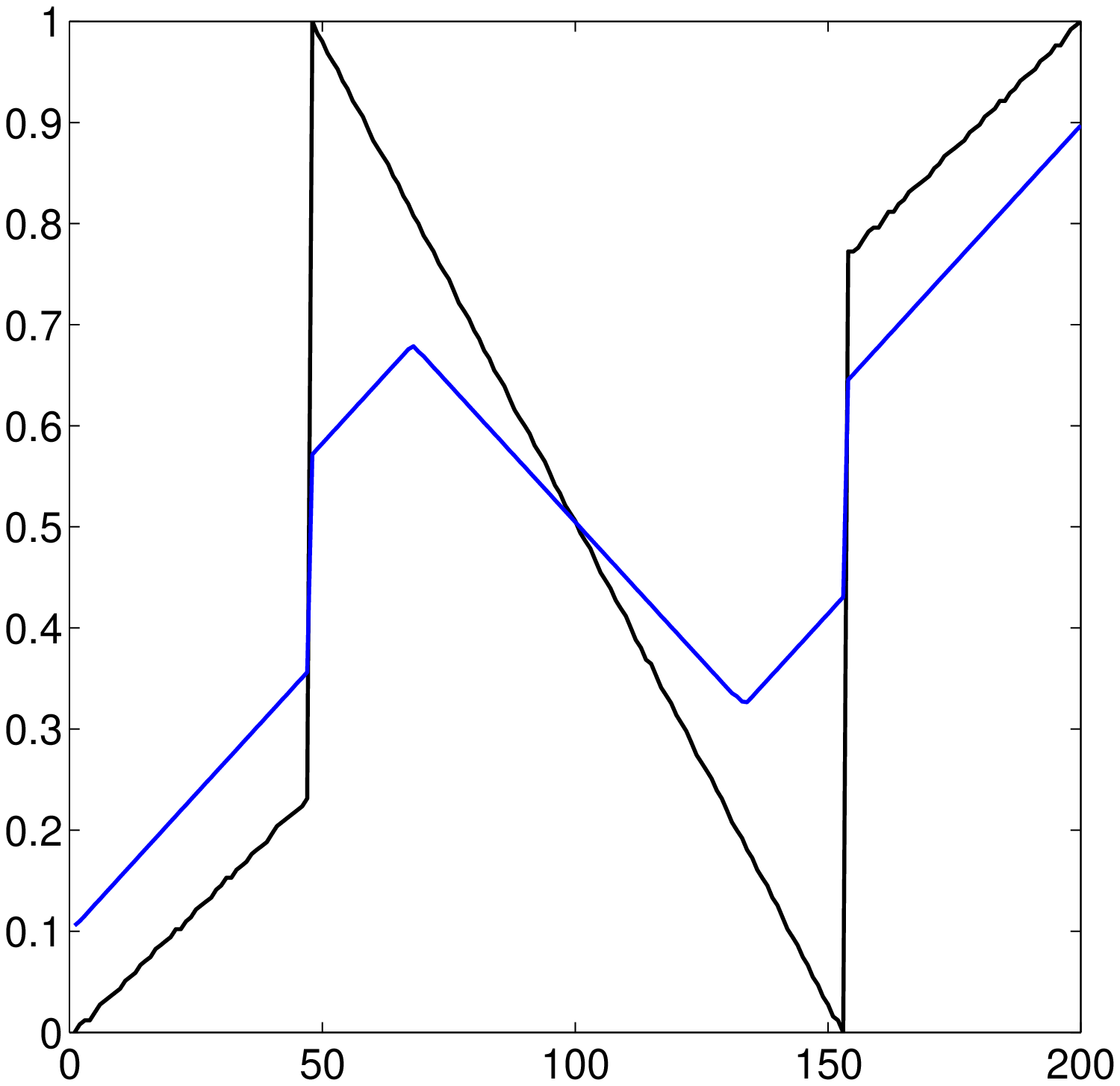}
                \caption{\centering Middle row profiles of Figure \ref{square_infty:a} (blue) and the ground truth (black)}
                \label{square_infty:b}
\end{subfigure}
\begin{subfigure}[t]{0.193\textwidth}
                \centering                                                  
                \includegraphics[width=0.9\textwidth]{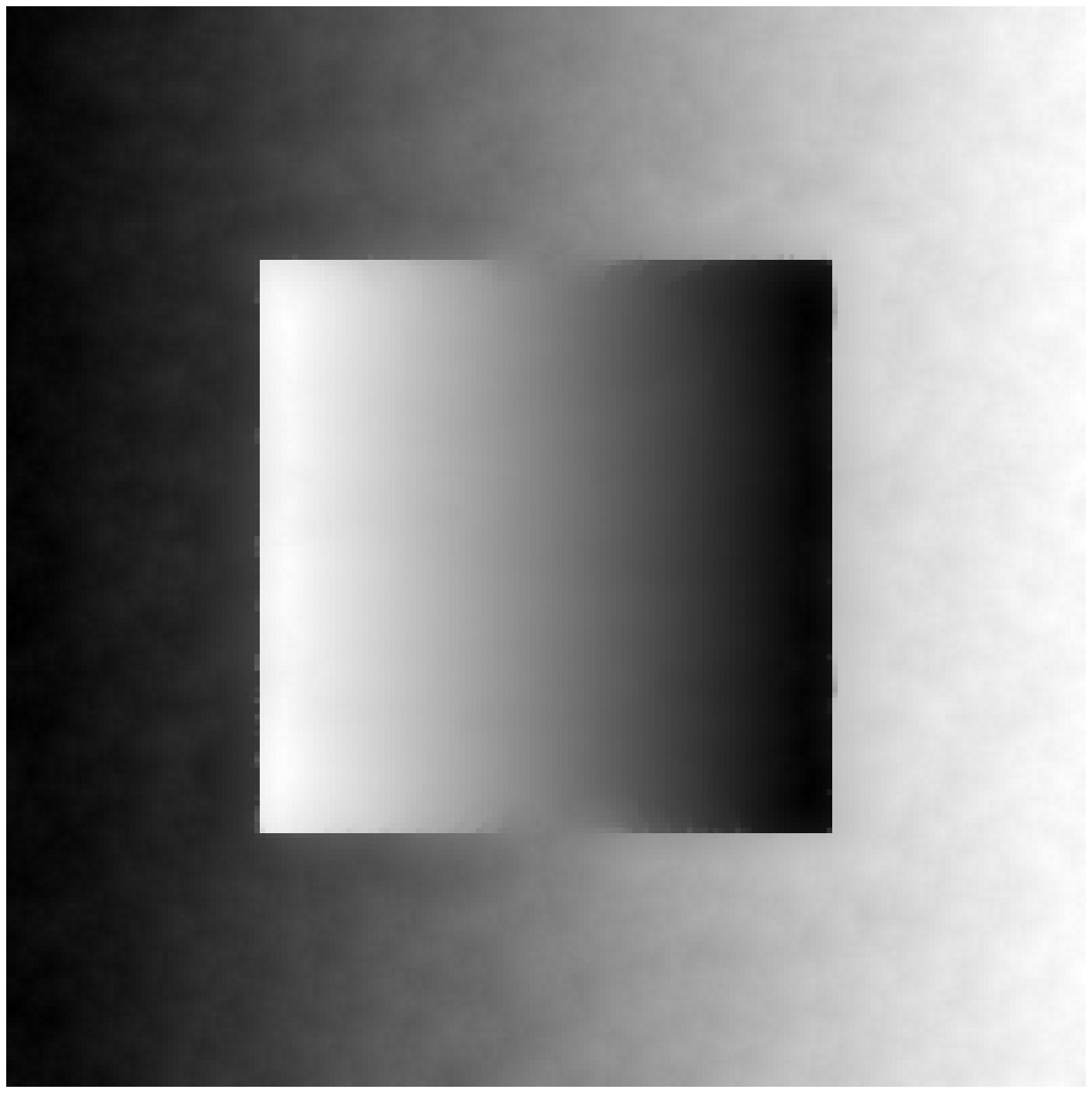}
               \caption{\centering Bregman iteration $\mathrm{TVL}^{\infty}$: $\alpha=5$, $\beta=60000$, $				\mathrm{SSIM}=0.9601$, }
                 \label{square_infty:c}
\end{subfigure}
\begin{subfigure}[t]{0.193\textwidth}
                \centering                                                  
                \includegraphics[height=0.9\textwidth]{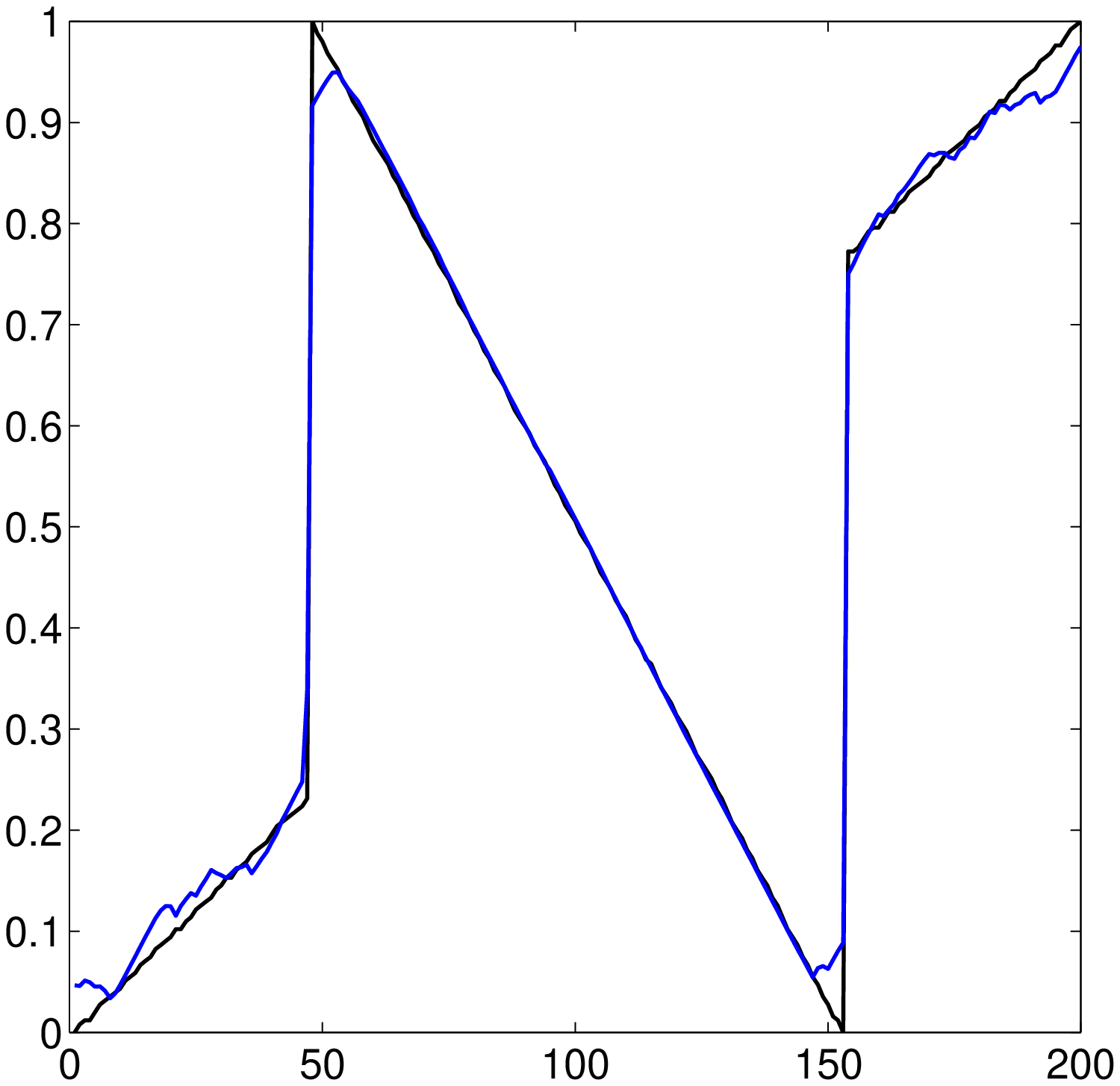}
               \caption{\centering Middle row profiles of Figure \ref{square_infty:c} (blue) and the ground truth (black) }
                 \label{square_infty:d}
\end{subfigure}\\[0.8cm]
\caption{Illustration of the fact that $\mathrm{TVL^{\infty}}$ regularisation favours gradients of fixed modulus }
\label{fig:square}
\end{center}\vspace{0.1cm}
\end{figure}
There we see that in order to get a staircasing--free result with $\mathrm{TVL^{\infty}}$ we also get a loss of geometric information, Figure \ref{square_infty:a}, as the model tries to fit an image with constant gradient, see  the middle row profiles in Figure \ref{square_infty:b}. While improved results can be achieved with the Bregman iteration version, Figure \ref{square_infty:c}, the result is not yet fully satisfactory as an \emph{affine staircasing} is now present in the image, Figure \ref{square_infty:d}.\\

\noindent
\textbf{Spatially adapted $\boldsymbol{\mathrm{TVL}^{\infty}}$}: One way to allow for different gradient values in the reconstruction, or in other words allow the variable $w$ to take different values, is to treat $\beta$ as a spatially varying parameter, i.e., $\beta=\beta(x)$. This leads to the spatially adapted version of $\mathrm{TVL}^{\infty}$:
\begin{equation}\label{tvlinf_sa}
\mathrm{TVL}_{\mathrm{s.a.}}^{\infty}(u)=\min_{w\in\mathrm{L}^{\infty}(\om)} \alpha \|Du-w\|_{\mathcal{M}}+\|\beta w\|_{\mathrm{L}^{\infty}(\om)}.
\end{equation}
The idea is to choose $\beta$ large in  areas where the gradient is expected to be small and vice versa, see Figure \ref{fig:square_sa} for a simple illustration.
\begin{figure}[t!]
\begin{center}
\begin{subfigure}[t]{0.193\textwidth}
                \centering                                                  
                \includegraphics[width=0.9\textwidth]{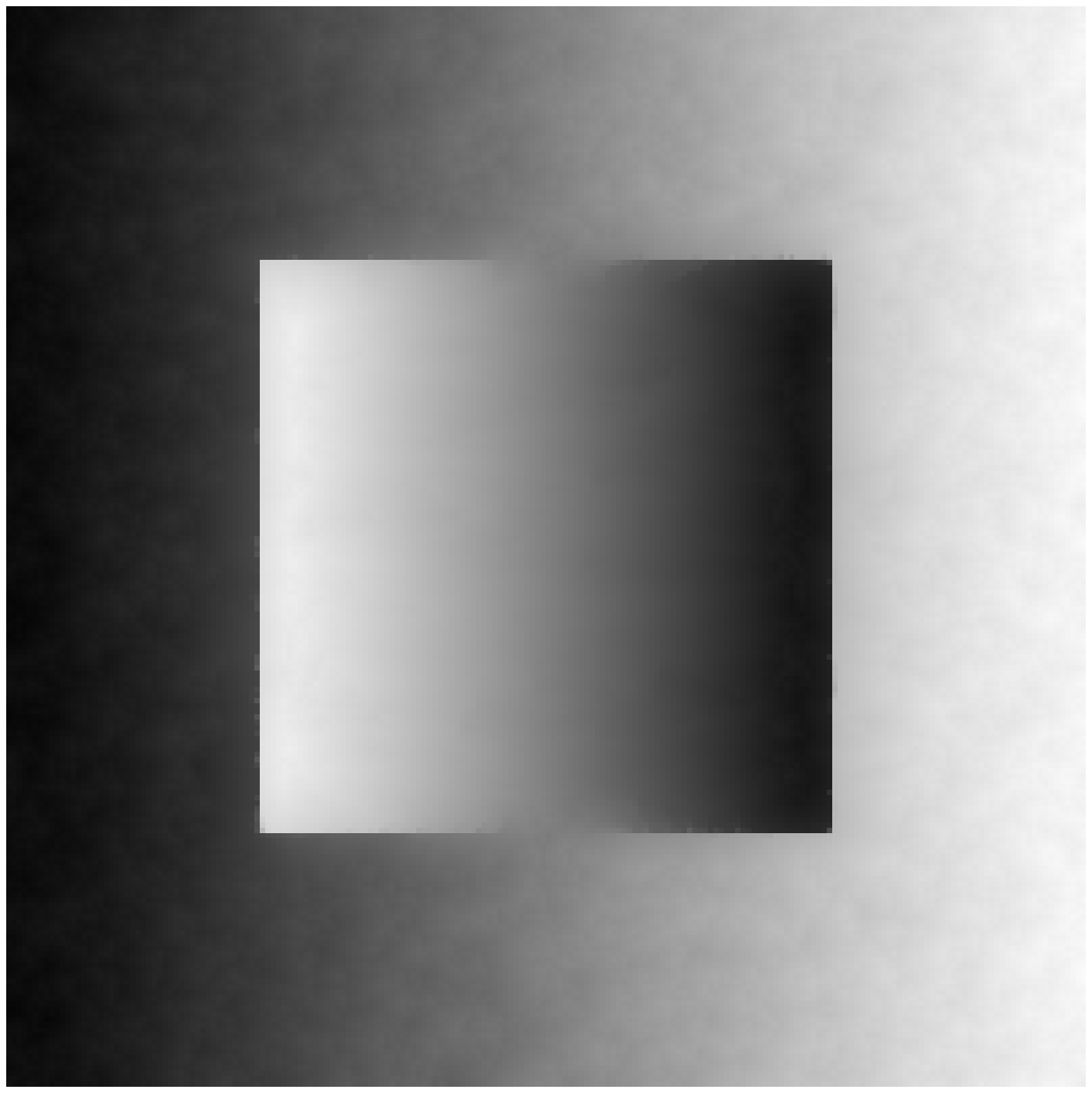}
                \caption{\centering $\mathrm{TVL}^{\infty}$: $\alpha=0.3$, $\beta=3500$, $\mathrm{SSIM}				=0.9547$}
                \label{weighted_tvl_inf:a}
\end{subfigure}
\begin{subfigure}[t]{0.193\textwidth}
                \centering                                                  
                \includegraphics[width=0.9\textwidth]{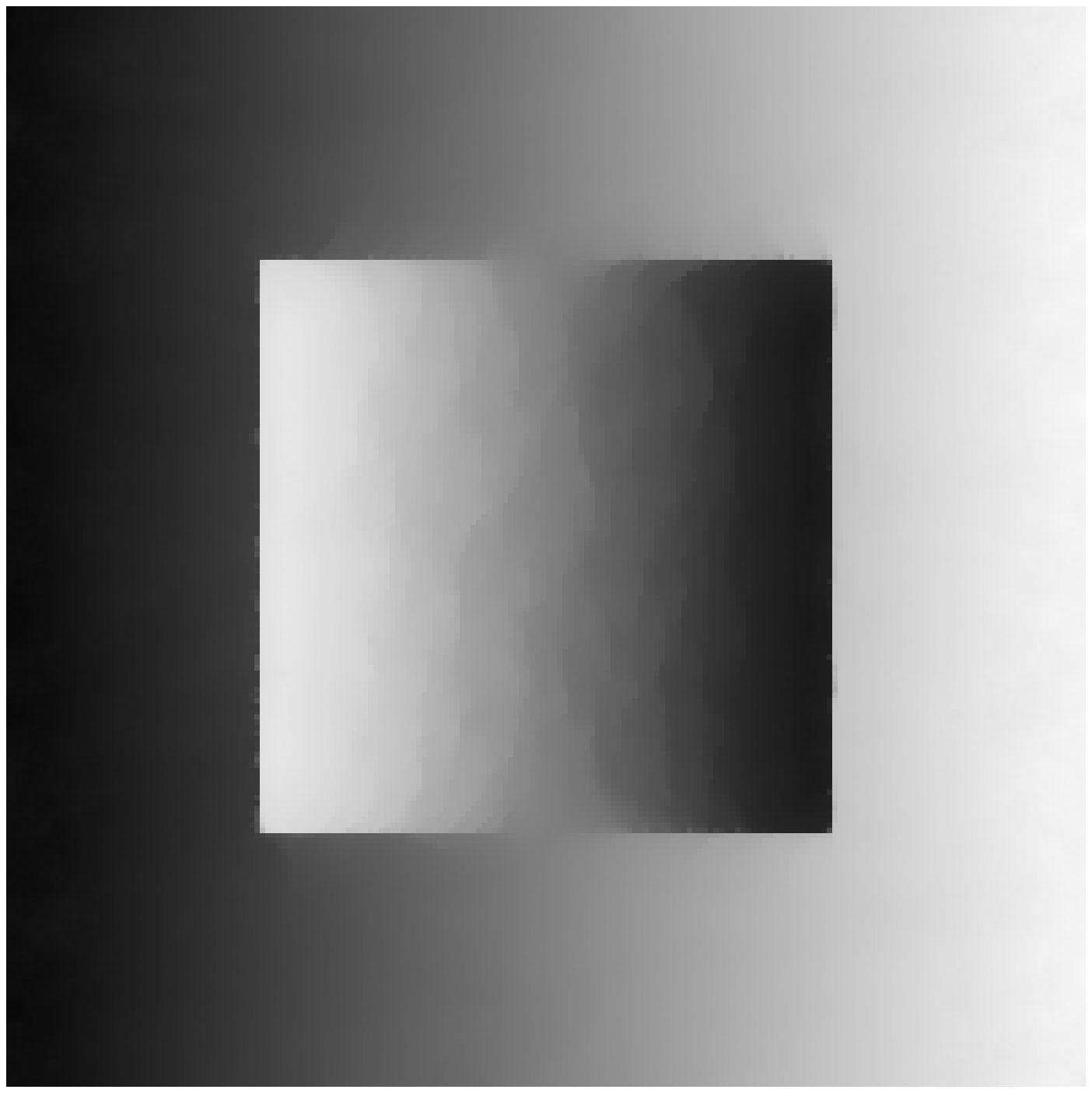}
                \caption{\centering $\mathrm{TVL}^{\infty}$: $\alpha=0.3$, $\beta=7000$, $\mathrm{SSIM}				=0.9672$}
                \label{weighted_tvl_inf:b}
\end{subfigure}
\begin{subfigure}[t]{0.193\textwidth}
                \centering                                                  
                \includegraphics[width=0.9\textwidth]{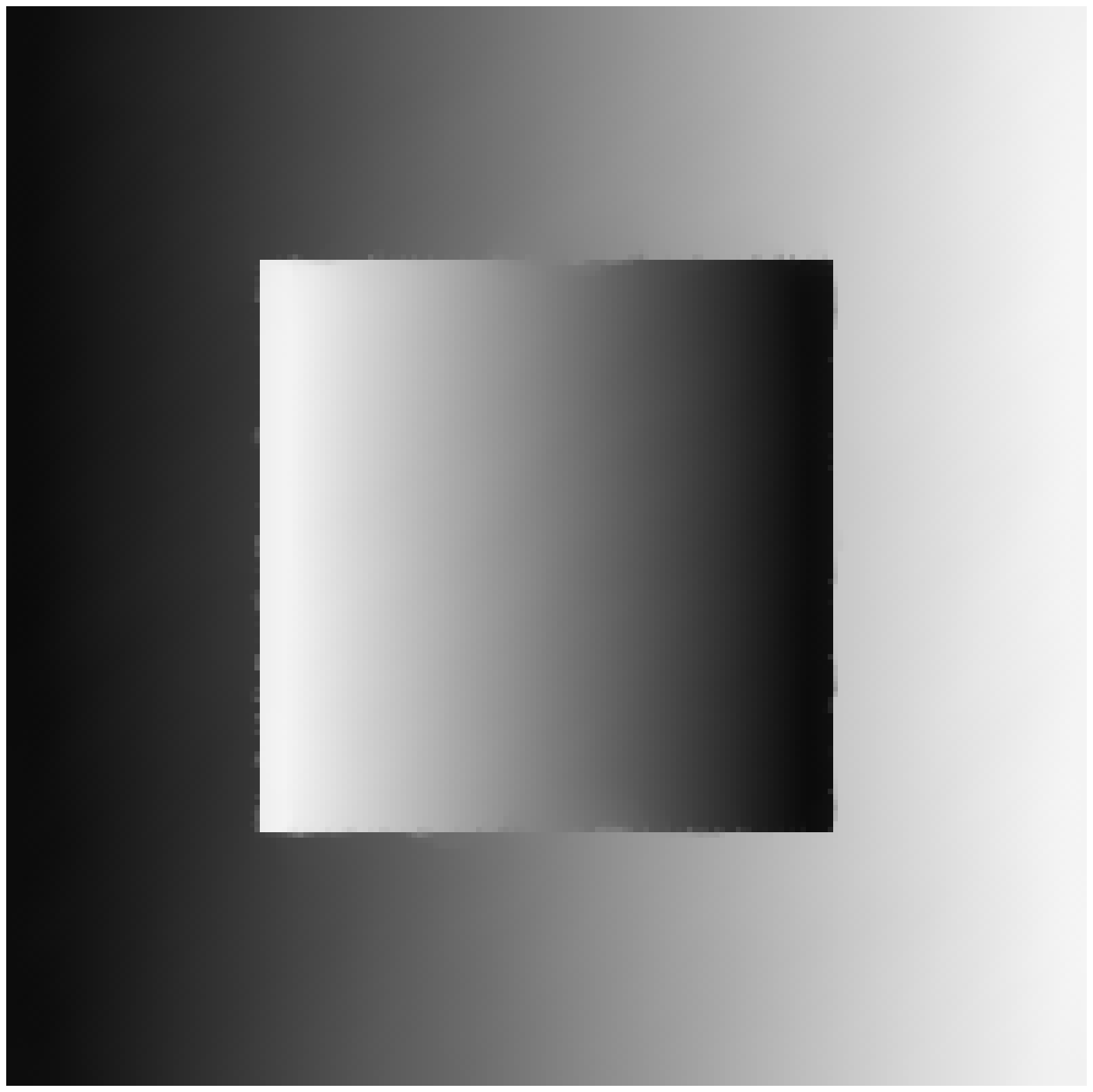}
                 \caption{\centering Bregman iteration $\mathrm{TGV}$: $\alpha=2$, $\beta=10$, $							\mathrm{SSIM}=0.9889$} 
                \label{weighted_tvl_inf:d}
\end{subfigure}
\begin{subfigure}[t]{0.193\textwidth}
                \centering                                                  
                \includegraphics[width=0.9\textwidth]{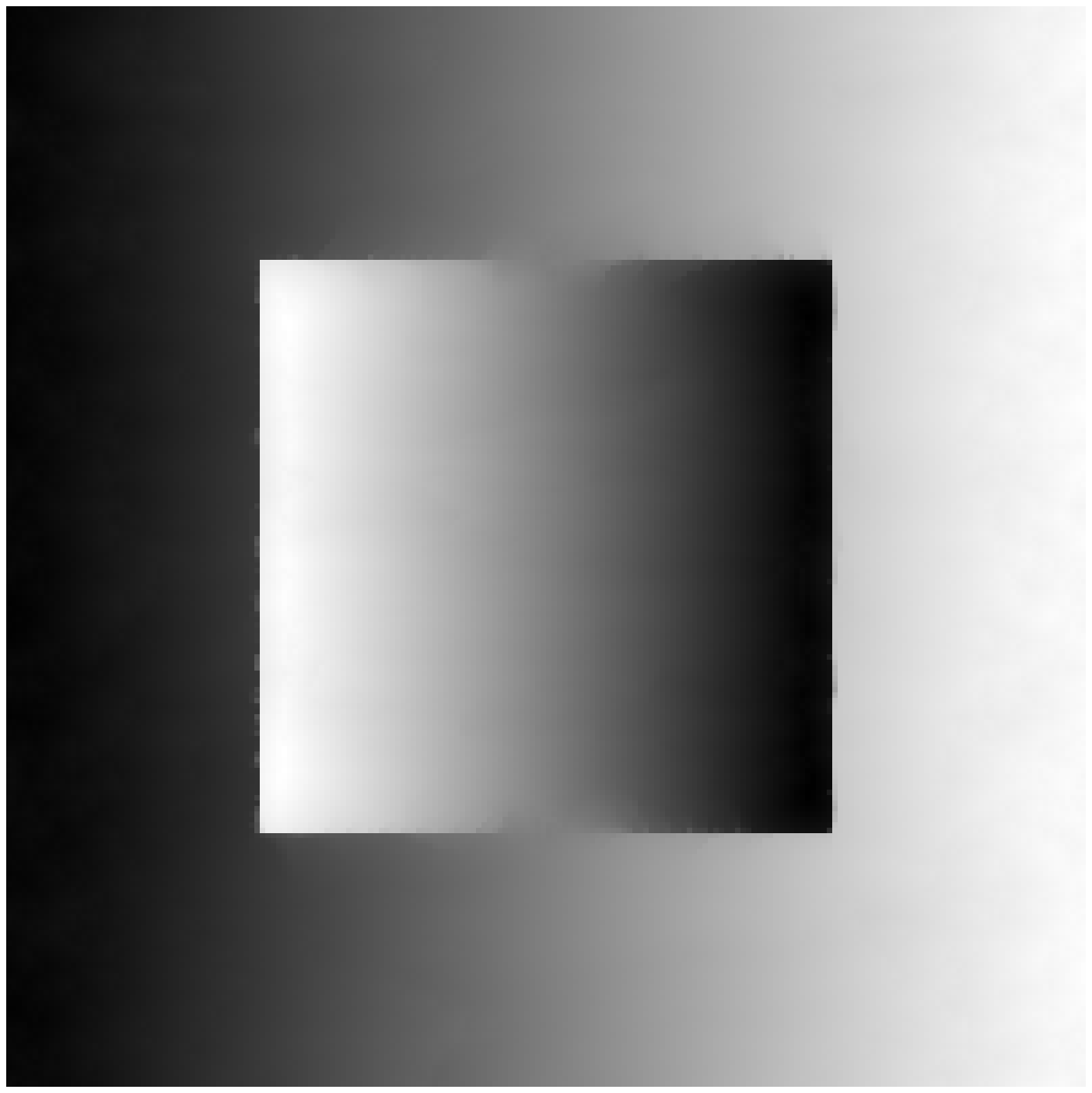}
                \caption{\centering Bregman iteration $\mathrm{TVL}_{\mathrm{s.a.}}^{\infty}$: $\alpha=5$, $\beta_{in}	=6\cdot10^{4}$, $\beta_{out}=11\cdot10^{4}$, $\mathrm{SSIM}=0.9837$}
                \label{weighted_tvl_inf:c}
\end{subfigure}
\begin{subfigure}[t]{0.193\textwidth}
                \centering                                                  
                \includegraphics[width=0.9\textwidth]{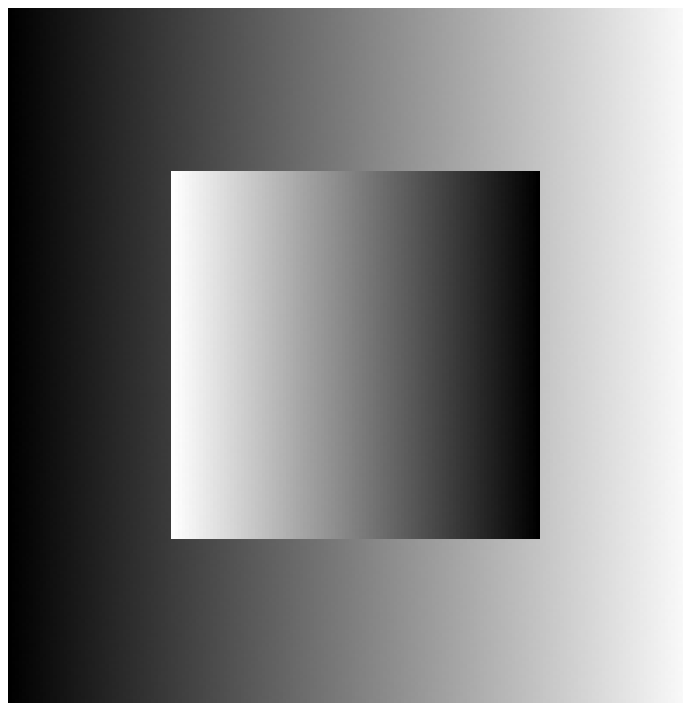}
                 \caption{\centering Ground truth} 
                \label{weighted_tvl_inf:clean}
\end{subfigure}
\\[0.8cm]
\caption{$\mathrm{TVL}^{\infty}$ reconstructions for different values of $\beta$. The best result is obtained by  spatially varying $\beta$, setting it inversely proportional to the gradient, where we obtain a similar result to the $\mathrm{TGV}$ one}
\label{fig:square_sa}
\end{center}
\end{figure}
In this example the slope inside the inner square is roughly twice the slope outside.
We can achieve a perfect reconstruction inside the square by setting $\beta=3500$, with artefacts outside, see Figure \ref{weighted_tvl_inf:a} and a perfect reconstruction outside by setting the value of $\beta$ twice as large, i.e., $\beta=7000$, Figure \ref{weighted_tvl_inf:b}. In that case, artefacts appear inside the square. By setting a spatially varying $\beta$ with a ratio $\beta_{out}/\beta_{in}\simeq 2$ and using the Bregman iteration version, we  achieve an almost perfect result, visually very similar to the $\mathrm{TGV}$ one, Figures \ref{weighted_tvl_inf:d}--\ref{weighted_tvl_inf:c}. This example suggests that ideally $\beta$ should be inversely proportional to the gradient of the ground truth. Since this information is not available in practice we use a pre-filtered version of the noisy image and we set
\begin{equation}\label{beta_rule}
\beta(x)=\frac{c}{|\nabla f_{\sigma}(x)|+\epsilon}.
\end{equation} 
Here $c$ is a positive constant to be tuned, $\epsilon>0$ is a small parameter and $f_{\sigma}$ denotes a smoothing of the data $f$ with a Gaussian kernel.
\begin{figure}[t!]
\begin{center}
\begin{subfigure}[t]{0.24\textwidth}
                \centering                                                  
                \includegraphics[width=0.9\textwidth]{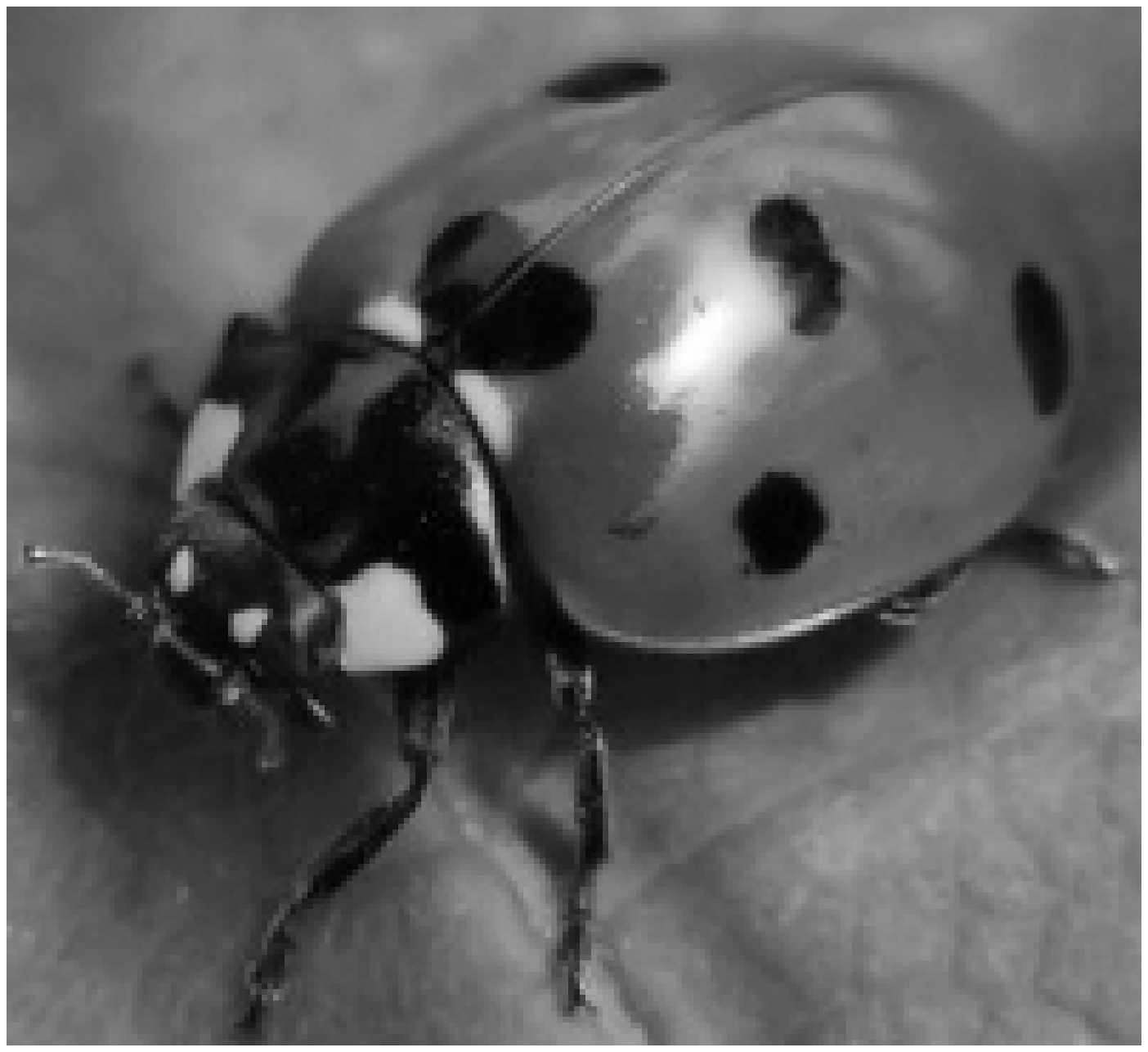}
                \caption{\centering Ladybug}
                \label{ladybug_1:a}
\end{subfigure}
\begin{subfigure}[t]{0.24\textwidth}
                \centering                                                  
                \includegraphics[width=0.9\textwidth]{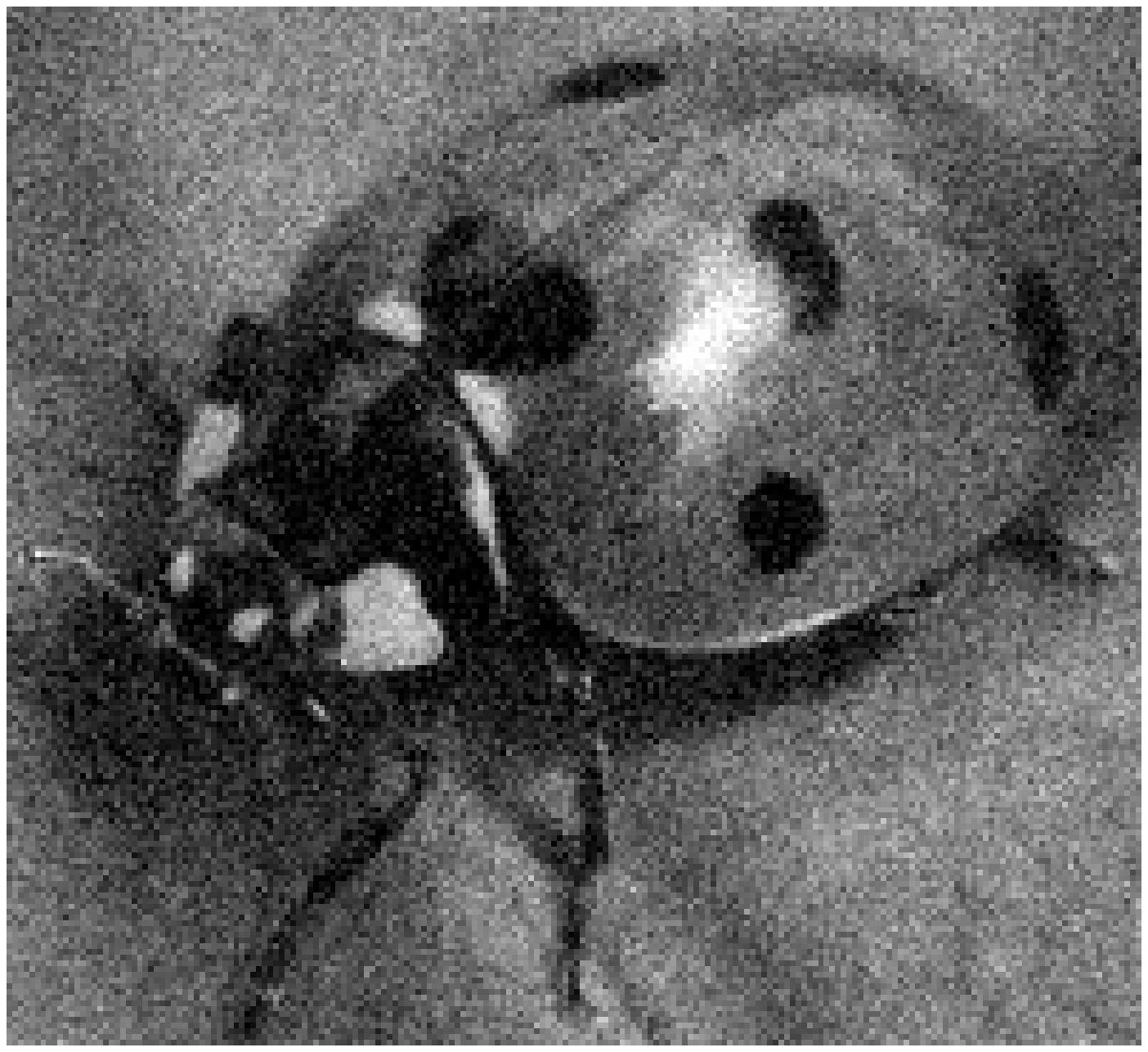}
                \caption{\centering Noisy data, Gaussian noise of $\sigma=0.005$, $\mathrm{SSIM}					=0.4076$}
                \label{ladybug_1:b}
\end{subfigure}
\begin{subfigure}[t]{0.24\textwidth}
                \centering                                                  
                \includegraphics[width=0.9\textwidth]{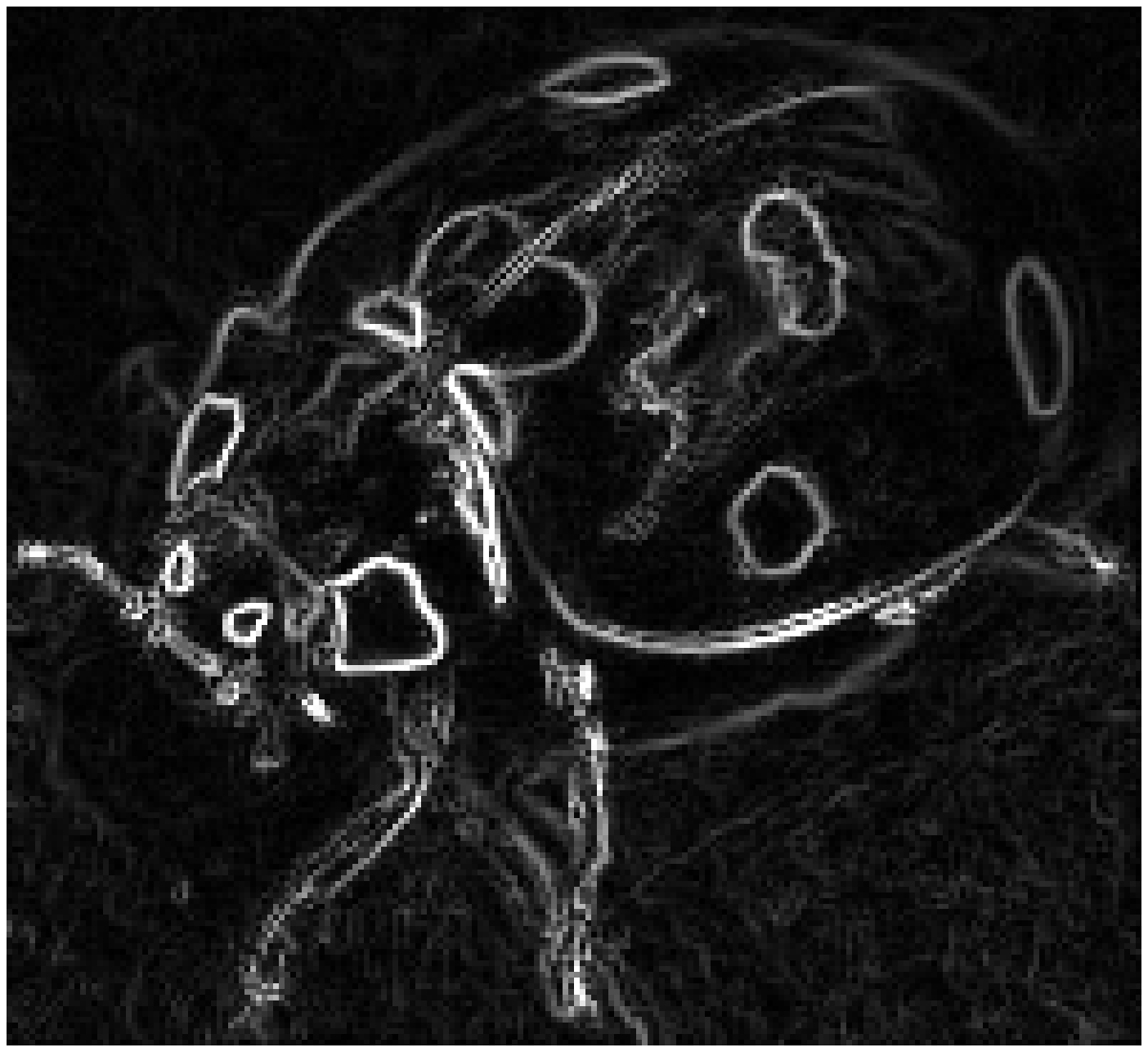}
                \caption{\centering Gradient of the ground truth}
                 \label{ladybug_1:d}
\end{subfigure}
\begin{subfigure}[t]{0.24\textwidth}
                \centering 
                 \includegraphics[width=0.9\textwidth]{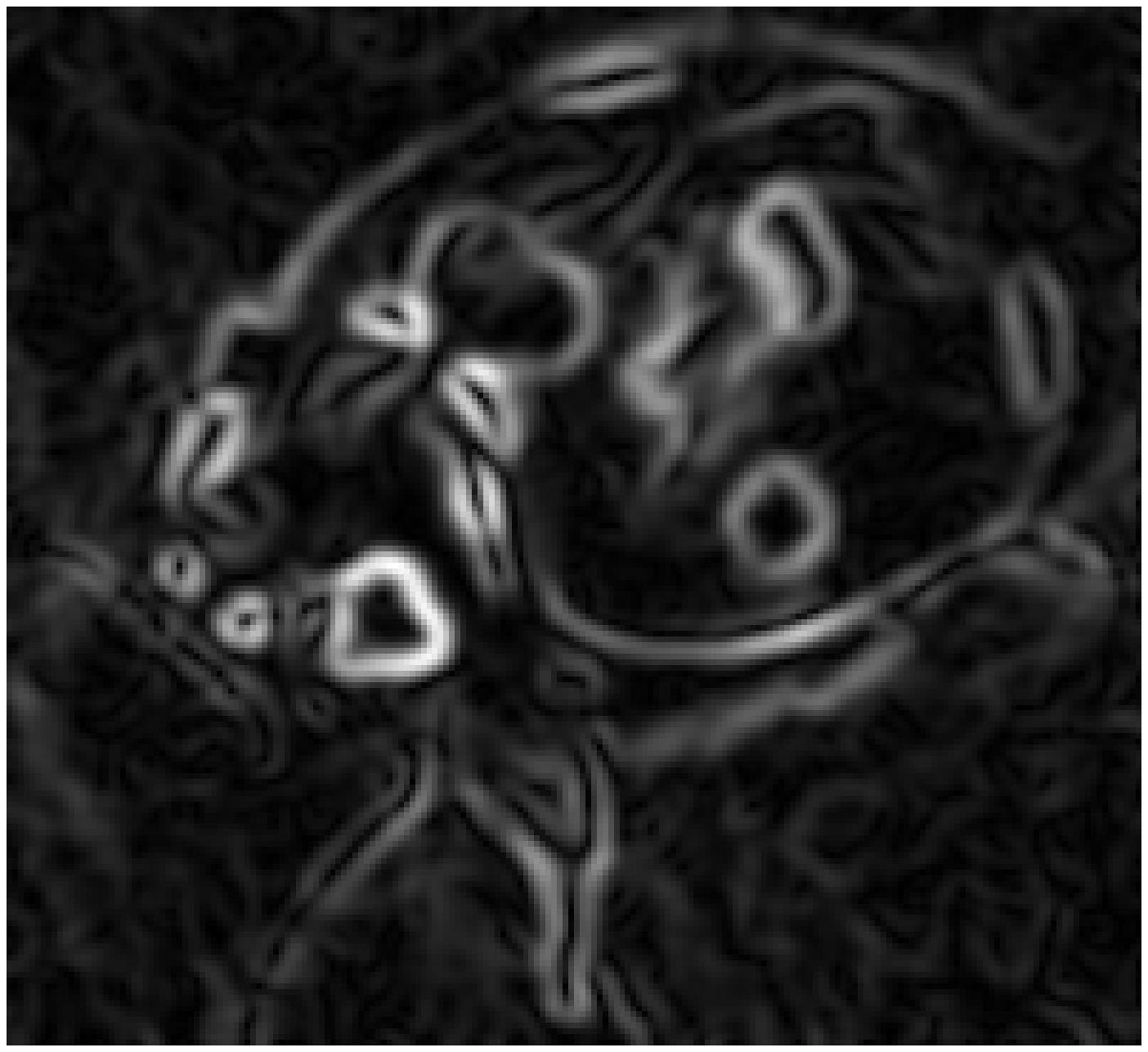}                                               		       \caption{\centering Gradient of the smoothed data,  $\sigma=2$, 13x13 pixels window}
                \label{ladybug_1:e}
\end{subfigure}\\[1cm]
\begin{subfigure}[t]{0.24\textwidth}
                \centering                                                  
                \includegraphics[width=0.9\textwidth]{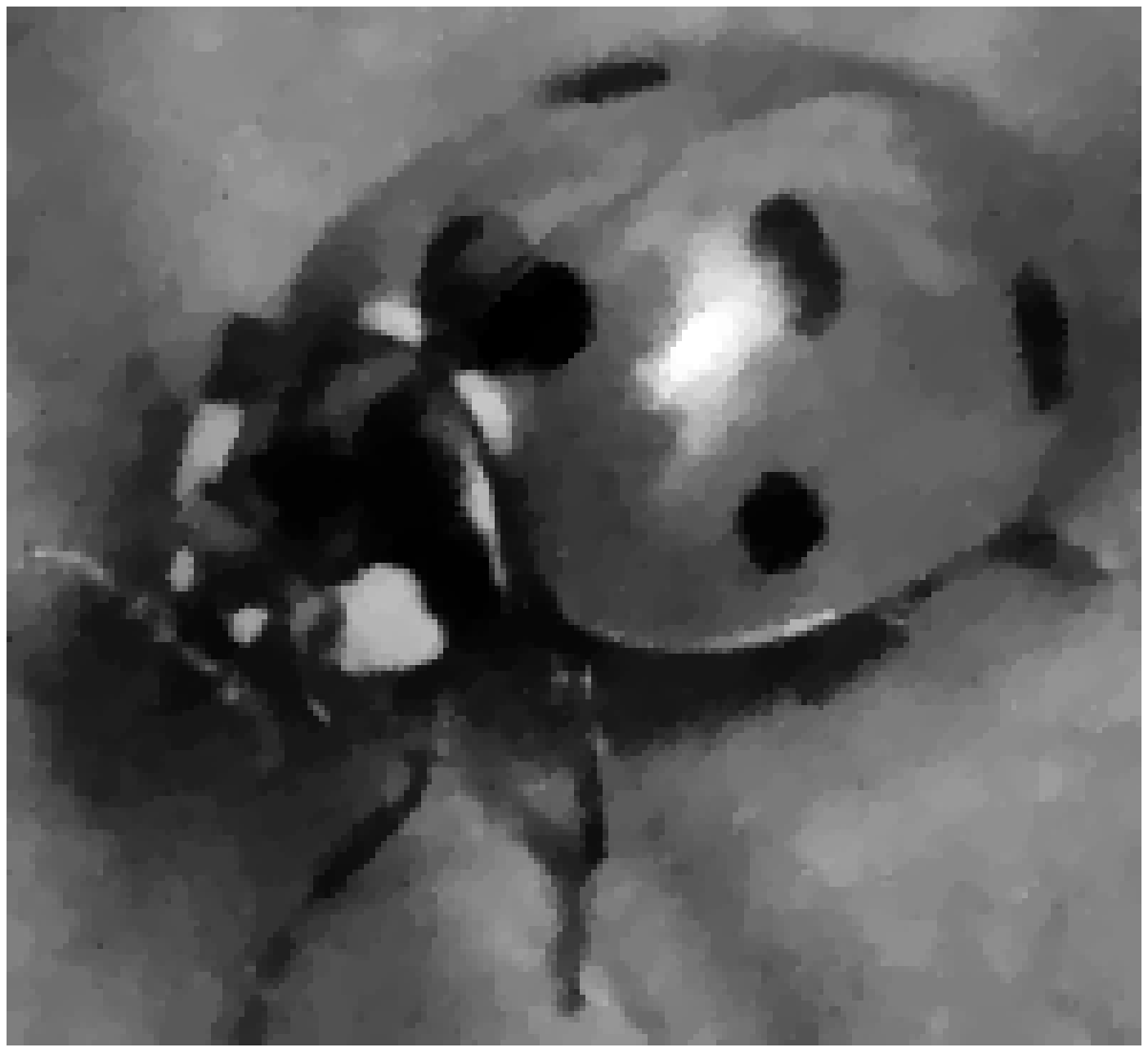}
                \caption{\centering $\mathrm{TV}$: $\alpha=0.06$, $\mathrm{SSIM}=0.8608$}
                \label{ladybug_2:f}
\end{subfigure}
\begin{subfigure}[t]{0.24\textwidth}
                \centering                                                  
                \includegraphics[width=0.9\textwidth]{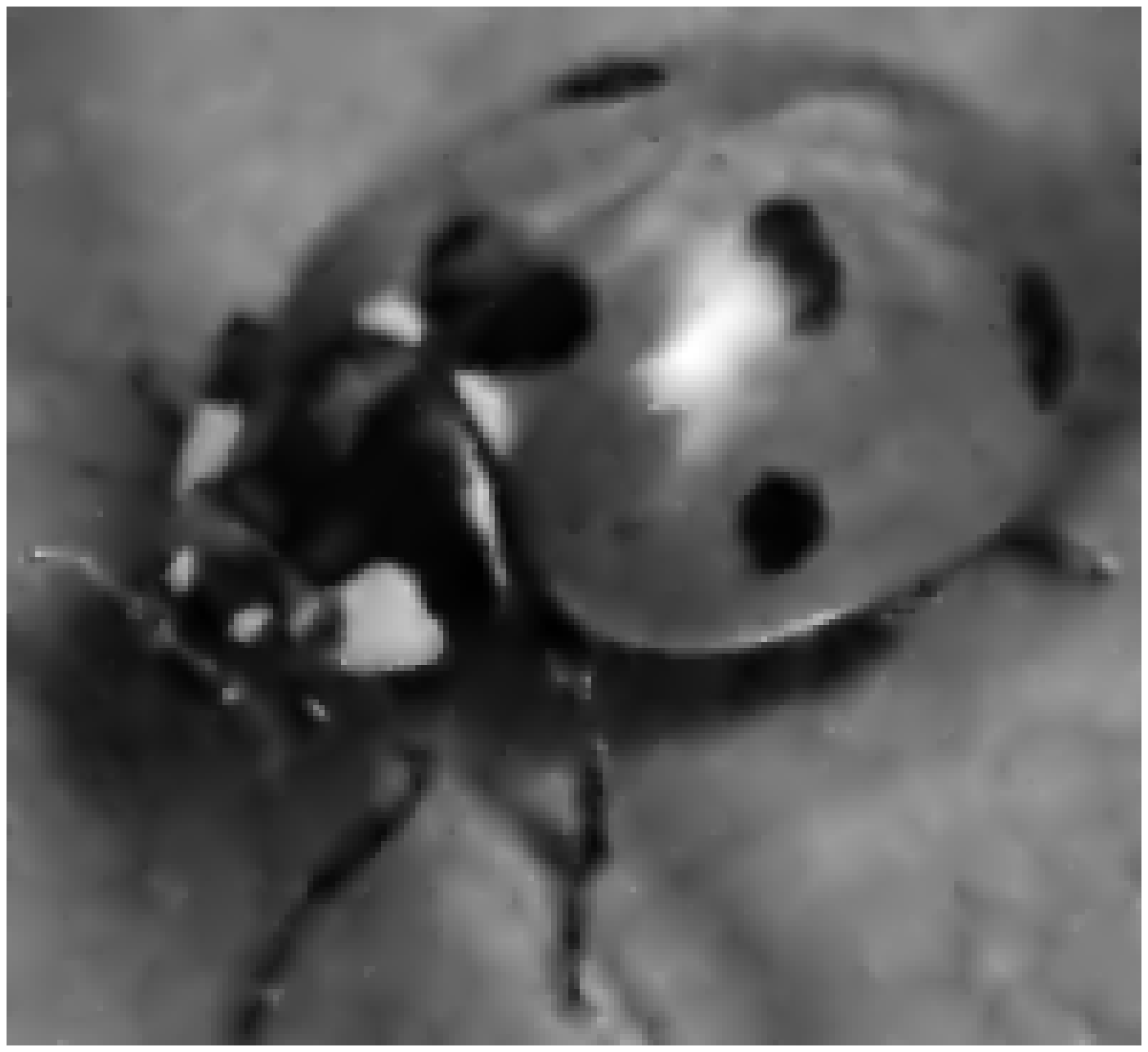}
                \caption{\centering $\mathrm{TGV}$: $\alpha=0.068$, $\beta=0.046$, $\mathrm{SSIM}				=0.8874$}
                \label{ladybug_2:g}
\end{subfigure}
\begin{subfigure}[t]{0.24\textwidth}
                \centering                                                  
                \includegraphics[width=0.9\textwidth]{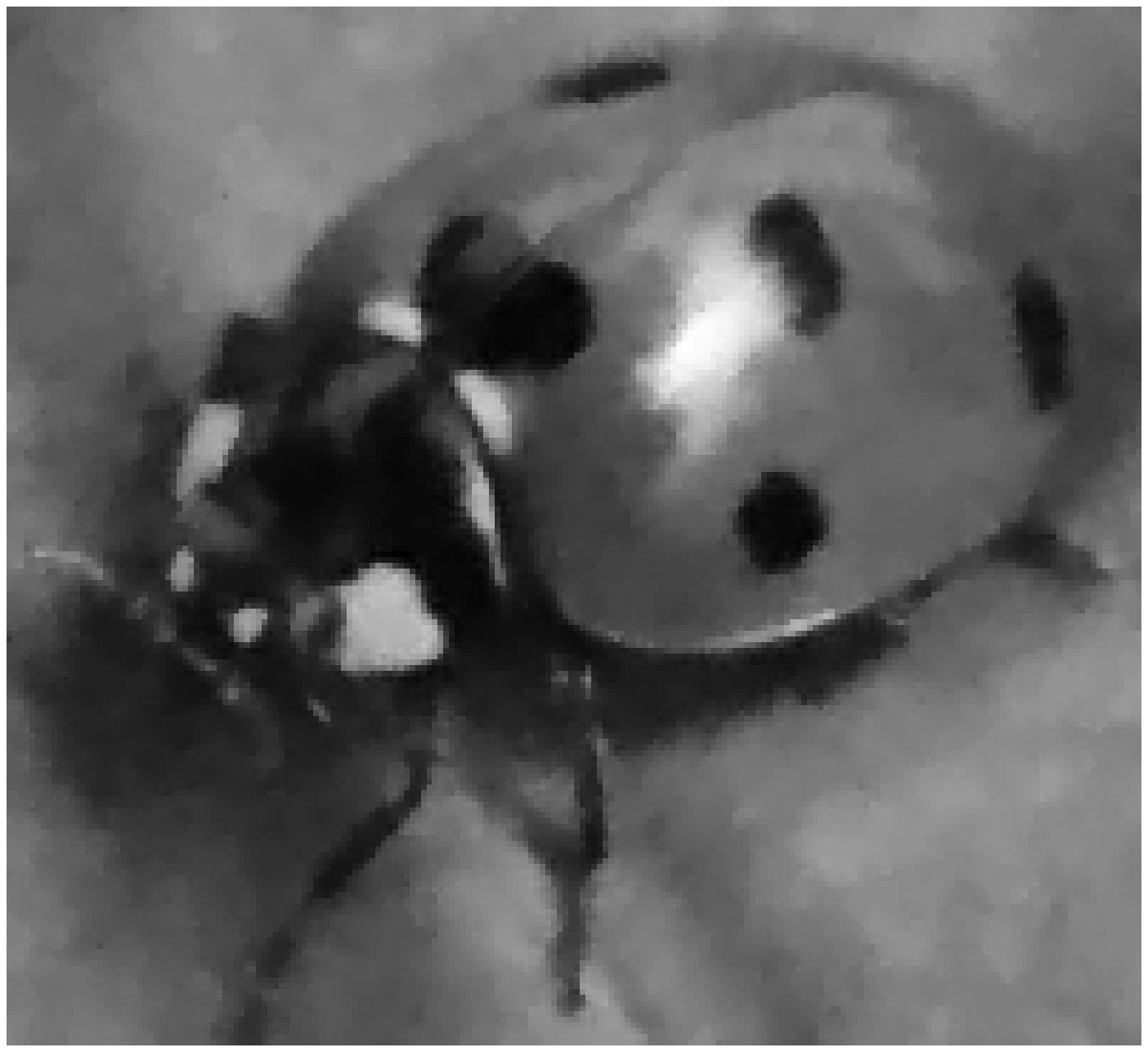}
                \caption{\centering $\mathrm{TVL}_{\mathrm{s.a.}}^{\infty}$: $\alpha=0.07$ and $\beta$ computed from filtered version with $c=30$, $\varepsilon=10^{-4}$, $\mathrm{SSIM}=0.8729$}
                \label{ladybug_2:h}
\end{subfigure}
\begin{subfigure}[t]{0.24\textwidth}
                \centering                                                  
                \includegraphics[width=0.9\textwidth]{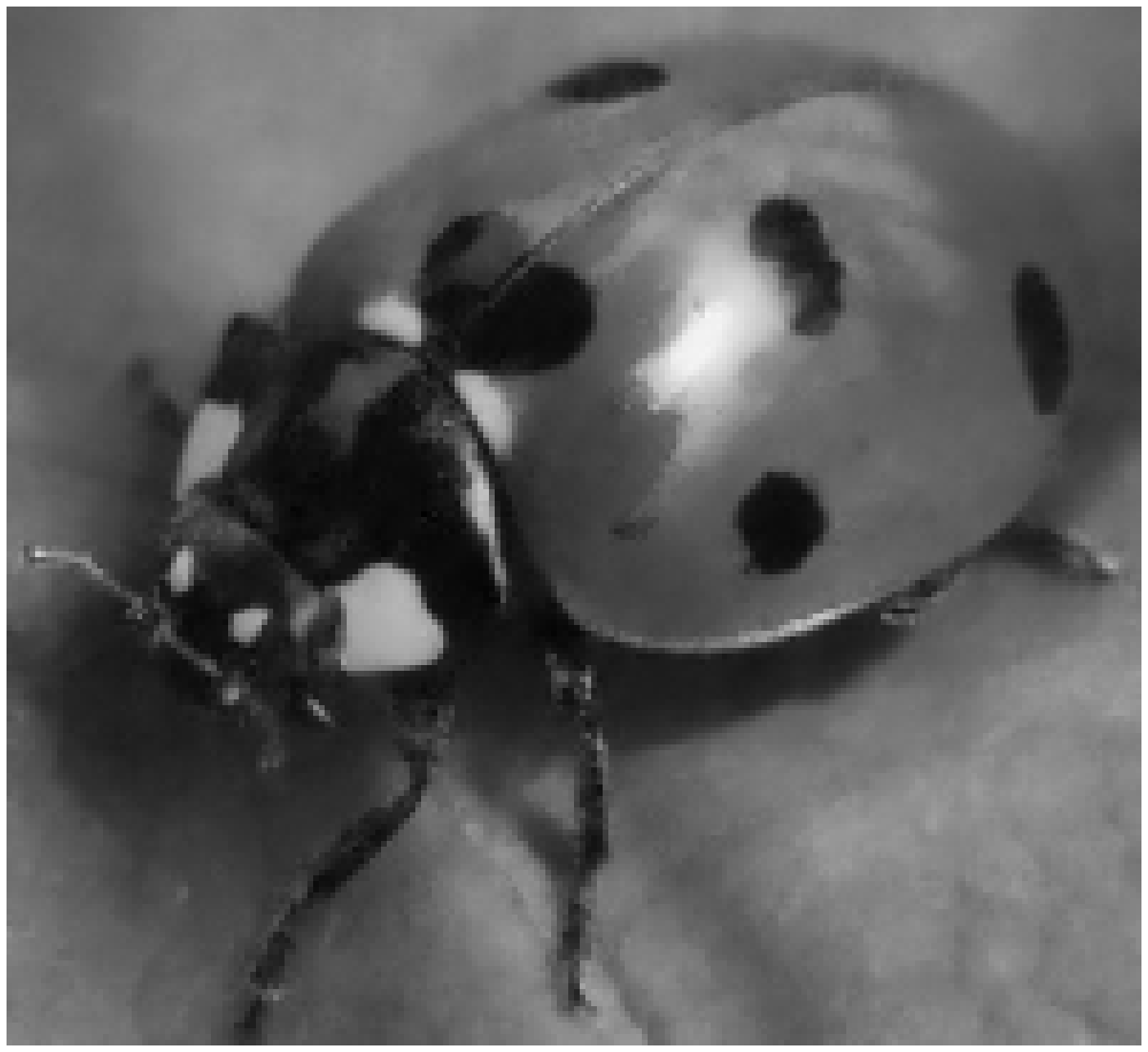}
                \caption{\centering $\mathrm{TVL}_{\mathrm{s.a.}}^{\infty}$: $\alpha=0.5$ and $\beta$ computed from ground truth with $c=50$, $\varepsilon=10^{-4}$, $\mathrm{SSIM}=0.9300$}
                \label{ch3_ladybug_2:d}
\end{subfigure}\\[0.8cm]
\end{center}
\caption{Best reconstruction of the ``Ladybug'' in terms of SSIM using $\mathrm{TV}$, $\mathrm{TVG}$ and spatially adapted $\mathrm{TVL^{\infty}}$ regularisation. The $\beta$ for the latter is computed both from the filtered (Figure \ref{ladybug_2:h}) and the ground truth image (Figure \ref{ch3_ladybug_2:d})}% \\[0.2cm]
\label{ladybug}
\end{figure}
We have applied the spatially adapted $\mathrm{TVL}^{\infty}$ (non-Bregman) with the rule \eqref{beta_rule} in the natural image ``Ladybug'' in Figure \ref{ladybug}. There we pre-smooth the noisy data with a discrete Gaussian kernel of $\sigma=2$, Figure \ref{ladybug_1:e}, and then apply  $\mathrm{TVL}_{\mathrm{s.a.}}^{\infty}$ with the rule \eqref{beta_rule}, Figure \ref{ladybug_2:h}. Comparing to the best $\mathrm{TGV}$ result in Figure \ref{ladybug_2:g}, the SSIM value is slightly smaller but there is a better recovery of the image details (objective). Let us note here that we do not claim that our rule for choosing $\beta$ is the optimal one. For demonstration purposes, we  show a reconstruction where we have computed $\beta$ using the gradient of ground truth $u_{\mathrm{g.t.}}$, Figure \ref{ladybug_1:d}, as $\beta(x)=c/(|\nabla u_{\mathrm{g.t.}}(x)|+\epsilon)$, with excellent results, Figure \ref{ch3_ladybug_2:d}. This is of course impractical, since the gradient of the ground truth is typically not available but it shows that there is plenty of room for improvement regarding the choice of $\beta$. One could also think of reconstruction tasks where a good quality version of the gradient of the image is available, along with a noisy version of the image itself. Since the purpose of the present paper is to demonstrate the capabilities of the $\mathrm{TVL}^{\infty}$ regulariser, we leave that for future work. 

\subsubsection*{Acknowledgments.} The authors acknowledge support of the Royal Society International Exchange Award Nr. IE110314. This work is further supported by the King Abdullah University for Science and Technology (KAUST) Award No. KUK-I1-007-43, the EPSRC first
grant Nr. EP/J009539/1 and the EPSRC grant Nr. EP/M00483X/1.
 MB acknowledges further support by ERC via Grant EU FP 7-ERC Consolidator Grant 615216 LifeInverse. KP acknowledges the financial support of EPSRC and the Alexander von Humboldt Foundation while in UK and Germany respectively. EP acknowledges support by Jesus College, Cambridge and Embiricos Trust.

\bibliographystyle{abbrv}
\bibliography{kostasbib}

\begin{thebibliography}{10}

\bibitem{AmbrosioBV}
L.~Ambrosio, N.~Fusco, and D.~Pallara.
\newblock {\em {Functions of bounded variation and free discontinuity
  problems}}.
\newblock Oxford University Press, USA, 2000.

\bibitem{TGVbregman}
M.~Benning, C.~Brune, M.~Burger, and J.~M{\"u}ller.
\newblock Higher-order {TV} methods -- {E}nhancement via {B}regman iteration.
\newblock {\em J. Sci. Comput.}, 54(2-3):269--310, 2013.

\bibitem{TGV}
K.~Bredies, K.~Kunisch, and T.~Pock.
\newblock Total generalized variation.
\newblock {\em SIAM J. Imaging Sci.}, 3(3):492--526, 2010.

\bibitem{BrediesL1}
K.~Bredies, K.~Kunisch, and T.~Valkonen.
\newblock Properties of {L}$^1$-{TGV}$^{\,2}$ : The one-dimensional case.
\newblock {\em Journal of Mathematical Analysis and Applications}, 398(1):438
  -- 454, 2013.

\bibitem{tvlp}
M.~Burger, K.~Papafitsoros, E.~Papoutsellis, and C.-B. Sch\"onlieb.
\newblock Infimal convolution regularisation functionals of {BV} and
  $\mathrm{L}^{p}$ spaces. {P}art {I}: The finite $p$ case.
\newblock {\em submitted, arXiv:1504.01956}.

\bibitem{Huber}
P.~Huber.
\newblock Robust regression: Asymptotics, conjectures and monte carlo.
\newblock {\em The Annals of Statistics}, 1:799--821, 1973.

\bibitem{OBG}
S.~Osher, M.~Burger, D.~Goldfarb, J.~Xu, and W.~Yin.
\newblock An iterative regularization method for total variation-based image
  restoration.
\newblock {\em Multiscale Model. Simul.}, 4:460--489.

\bibitem{Papafitsoros_Bredies}
K.~Papafitsoros and K.~Bredies.
\newblock A study of the one dimensional total generalised variation
  regularisation problem.
\newblock {\em Inverse Probl. Imaging}, 9(2):511--550, 2015.

\bibitem{papoutsellisphd}
E.~Papoutsellis.
\newblock {\em First-order gradient regularisation methods for image
  restoration. Reconstruction of tomographic images with thin structures and
  denoising piecewise affine images}.
\newblock PhD thesis, University of Cambridge, 2015.

\bibitem{rudin1992nonlinear}
L.~Rudin, S.~Osher, and E.~Fatemi.
\newblock Nonlinear total variation based noise removal algorithms.
\newblock {\em Phys. D}, 60(1-4):259--268.

\end{thebibliography}

\end{document}